\documentclass[10pt]{article}

\usepackage[margin=2.25cm]{geometry}

\usepackage{amsmath}
\usepackage{amssymb}
\usepackage{amsthm}
\usepackage{amsfonts}
\usepackage{enumerate}
\usepackage{graphicx}
\usepackage{url}
\usepackage[table]{xcolor}
\usepackage{soul}
\usepackage{color}

\usepackage{subcaption}

\usepackage{ytableau}

\usepackage{tikz}
\usetikzlibrary{arrows,shapes,positioning,decorations.markings,decorations.pathreplacing}
\tikzstyle arrowstyle=[scale=1.2]
\tikzstyle directed=[postaction={decorate,decoration={markings,
    mark=at position .7 with {\arrow[arrowstyle]{stealth}}}}]

\usepackage{array}
\newcolumntype{?}{!{\vrule width 3pt}}

\newtheorem{theorem}{Theorem}[section]

\newtheorem{corollary}[theorem]{Corollary}
\newtheorem{proposition}[theorem]{Proposition}
\newtheorem{conjecture}[theorem]{Conjecture}

\def\Z{\mathbb{Z}}

\DeclareMathOperator{\svt}{\mathbb{S}}
\DeclareMathOperator{\im}{\text{im}}
\DeclareMathOperator{\id}{\text{Id}}
\DeclareMathOperator{\pro}{\text{PRO}}
\DeclareMathOperator{\pc}{\text{PC}}

\definecolor{superlight}{rgb}{0.95,0.95,0.95}
\definecolor{LightGray}{rgb}{0.85,0.85,0.85}
\definecolor{Sage}{rgb}{0.75,0.75,0.75}

\usepackage{authblk}

\title{Set-Valued Young Tableaux\\and Product-Coproduct Prographs}
\author[1]{Paul Drube}
\author[2]{Maxwell Krueger}
\author[3]{Ashley Skalsky}
\author[4]{Meghan Wren}
\affil[1]{Valparaiso University}
\affil[2]{Muhlenberg College}
\affil[3]{Minnesota State University Moorhead}
\affil[4]{SUNY Brockport}

\date{VERUM 2017, Valparaiso University\\Research funded by NSF Grant DMS-1559912}

\begin{document}

\maketitle

\begin{abstract}
Standard set-valued Young tableaux are a generalization of standard Young tableaux where cells can contain unordered sets of integers, with the added condition that every integer at position $(i,j)$ must be smaller that every integer at both $(i+1,j)$ and $(i,j+1)$.  In this paper, we explore properties of standard set-valued Young tableaux with three rows and a fixed number of integers in every cell of each row (referred to as set-valued tableaux with row-constant density).  Our primary focus is on standard set-valued Young tableaux with $1$ integer in each first-row cell, $k-1$ integers in each second-row cell, and $1$ integer in each third-row cell.  For rectangular shapes $\lambda=n^3$, such tableaux are placed in bijection with closed $k$-ary product-coproduct prographs: directed plane graphs that correspond to finite compositions involving a $k$-ary product operator and a $k$-ary coproduct operator.  That bijection is extended to three-row set-valued Young tableaux of non-rectangular and skew shape, and it is shown that a set-valued analogue of the Sch\"utzenberger involution on tableaux corresponds to $180$-degree rotation of the associated prographs.  As a set-valued analogue of the hook-length formula is currently lacking, we also present direct enumerations of three-row standard set-valued Young tableaux for a variety of row-constant densities and a small number of columns.  We then argue why the numbers of tableaux with the row-constant density $(1,k-1,1)$ should be interpreted as a one-parameter generalization of the three-dimensional Catalan numbers that mirrors the generalization of the (two-dimensional) Catalan numbers provided by the $k$-Catalan numbers. 
\end{abstract}

\section{Introduction: Standard Set-Valued Young Tableaux}
\label{sec: intro}

Consider a non-increasing sequence of positive integers $\lambda = (\lambda_1,\hdots,\lambda_m)$ that sum to $N$.  Following the English notation, a Young diagram $Y$ of shape $\lambda$ is a left-justified array of $N$ cells with $\lambda_i$ cells in the $i^{th}$ row from the top of the array.  Given a Young diagram of shape $\lambda$, a Young tableau of that shape is a bijection from the set of integers $[N]=\{1,\hdots,N\}$ to the cells of $Y$.  For a Young tableau to be a standard Young tableau, the entries in the tableau must increase left to right across each row and top to bottom down each column.  We denote the set of standard Young tableaux of shape $\lambda$ as $S(\lambda)$, and adopt the shorthand notation of $S(n^m)$ in the case of the $m$-row rectangular shape $\lambda = (n,\hdots,n)$.  For a thorough introduction to Young tableaux, see Fulton \cite{Fulton}.

The number of standard Young tableaux of arbitrary shape $\lambda$ may be directly calculated using the hook-length formula, as originally given by Frame, Robinson and Thrall \cite{FRT}.  A quick application of the hook-length formula to the case of $\lambda = (n,n)$ yields the well-known identity that $\vert S(n^2) \vert = C_n= \frac{(2n)!}{n! (n+1)!}$, the $n^{th}$ Catalan number.  Generalizing to the $d$-row rectangular case of $\lambda = (n,\hdots,n)$ similarly yields $\vert S(n^d) \vert = C_{d,n}$, where $C_{d,n} = \frac{(d-1)!(dn)!}{n! (n+1)! \hdots (n+d-1)!}$ is the $n^{th}$ $d$-dimensional Catalan number. 

Given an non-increasing sequence of positive integers $\lambda = (\lambda_1,\hdots,\lambda_{m_1})$ of $N_1$ and a non-increasing sequence of positive integers $\mu = (\mu_1,\hdots,\mu_{m_2})$ of $N_2$, where $0 \leq \mu_i \leq \lambda_i$ for all $i$, one can also define a skew Young diagram of shape $\lambda / \mu$ by removing the $\mu_i$ leftmost cells in the $i^{th}$ row of the Young diagram of shape $\lambda$, for all $1 \leq i \leq m$.  A skew Young tableau of shape $\lambda / \mu$ is a bijection from $[N_1 - N_2]$ to the cells of the skew Young diagram of shape $\lambda / \mu$.  Such a tableau is said to be a standard skew Young tableau if its entries increase left to right across each row and top to bottom down each column.  We denote the set of standard skew Young tableaux of shape $\lambda / \mu$ by $S(\lambda/\mu)$.

This paper is focused on a generalization of standard Young tableaux known as standard set-valued Young tableaux.  Consider a non-increasing sequence of positive integers $\lambda = (\lambda_1,\hdots,\lambda_m)$ and a sequence of positive integers $\rho = (\rho_1,\hdots,\rho_m)$ such that $\sum_{i=1}^m \lambda_i \rho_i = M$.  A \textbf{set-valued Young tableau} of shape $\lambda$ and (row-constant) \textbf{density} $\rho$ is a function from $[M]$ to the cells of the Young diagram $Y$ of shape $\lambda$ such that every cell in the $i^{th}$ row of $Y$ receives precisely $\rho_i$ integers.  The resulting tableau qualifies as a \textbf{standard set-valued Young tableau} if, for each cell $(i,j)$ of $Y$, every integer at position $(i,j)$ is smaller than every integer in the cells at $(i+1,j)$ and $(i,j+1)$.  In analogy with standard Young tableaux, we refer to these additional conditions as ``column-standardness" and ``row-standardness", respectively.  We denote the set of standard set-valued Young tableaux of shape $\lambda$ and density $\rho$ as $\svt(\lambda,\rho)$. See Figure \ref{fig: set-valued Young tableaux} for a collection of standard set-valued Young tableaux with $\lambda=(3,3)=3^2$ and $\rho=(2,1)$.  Given a skew Young diagram of shape $\lambda / \mu$, one may similarly define a \textbf{standard skew set-valued Young tableau} of shape $\lambda / \mu$ and (row-constant) density $\rho$.  We denote the set of such skew tableau by $\svt(\lambda / \mu, \rho)$.

\begin{figure}[ht!]
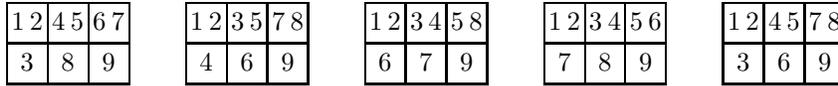

\centering
\begin{ytableau}
1 $\kern+2pt$ 2 & 4 $\kern+2pt$ 5 & 6 $\kern+2pt$ 7\\
3 & 8 & 9
\end{ytableau}
\hspace{0.2in}
\begin{ytableau}
1 $\kern+2pt$ 2 & 3 $\kern+2pt$ 5 & 7 $\kern+2pt$ 8\\
4 & 6 & 9
\end{ytableau}
\hspace{0.2in}
\begin{ytableau}
1 $\kern+2pt$ 2 & 3 $\kern+2pt$ 4 & 5 $\kern+2pt$ 8\\
6 & 7 & 9
\end{ytableau}
\hspace{0.2in}
\begin{ytableau}
1 $\kern+2pt$ 2 & 3 $\kern+2pt$ 4 & 5 $\kern+2pt$ 6\\
7 & 8 & 9
\end{ytableau}
\hspace{0.2in}
\begin{ytableau}
1 $\kern+2pt$ 2 & 4 $\kern+2pt$ 5 & 7 $\kern+2pt$ 8\\
3 & 6 & 9
\end{ytableau}
\caption{Five of the twelve elements of $\svt(3^2,\rho)$ with $\rho=(2,1)$.}
\label{fig: set-valued Young tableaux}
\end{figure}

Set-valued Young tableaux were originally introduced by Buch \cite{Buch} to study the $K$-theory of Grassmannians.  Heubach, Li and Mansour \cite{HLM} later provided standard set-valued Young tableaux with $\lambda = n^2$ and row-constant density $\rho=(k-1,1)$ as one of their many combinatorial interpretations of the $k$-Catalan numbers $C_n^k = \frac{(kn)!}{(kn-n+1)!n!}$.  That work was directly expanded upon by Drube \cite{Drube}, who used standard set-valued Young tableaux with two rows to provide new combinatorial interpretations of the Raney numbers, the rational Catalan numbers, and the solution to the generalized tennis ball problem.  For recent usages of set-valued Young tableaux in a more algebraic setting, see Reiner, Tenner and Young \cite{RTY} and Monical \cite{Monical}.

It is important to emphasize the current lack of a set-valued analogue to the hook-length formula.  This makes the enumeration of $\svt(\lambda,\rho)$ for arbitrary $\lambda$ and $\rho$ an extremely challenging problem, and comprehensive attempts at counting standard set-valued Young tableaux of arbitrary density $\rho$ have only been attempted for two-row shapes $\lambda = (a,b)$.  See Drube \cite{Drube} for calculations of $\vert S(\lambda,\rho) \vert$ in the two-row rectangular case, enumerations that corresponded to various generalizations of the (two-dimensional) Catalan numbers.

With the two-row case relatively well-understood, this paper presents the first thorough investigation of standard set-valued Young tableaux in the case of three-row shapes.  Much as various choices of $\rho$ allow $\vert \svt(n^2,\rho) \vert$ to correspond to various generalizations of the two-dimensional Catalan numbers, $\vert \svt(n^3,\rho) \vert$ and certain choices of $\rho$ will correspond to various generalizations of the three-dimensional Catalan numbers.  To our knowledge, all three-dimensional Catalan generalizations discussed in this paper have yet to appear anywhere in the literature.

\subsection{Outline of Paper}

This paper proceeds as follows.  In Section \ref{sec: prographs}, we introduce $k$-ary product-coproduct prographs, a class of directed plane graphs that naturally extend existing combinatorial interpretations for both the two- and three-dimensional Catalan numbers.  Much of Section \ref{sec: prographs} may be seen as an and formalization of the work of Borie \cite{Borie}.  This motivates our focus on the sets $\svt(n^3,\rho)$ with $\rho=(1,k-1,1)$, which are placed in bijection with $k$-ary product-coproduct prographs (Theorem \ref{thm: prograph vs tableaux bijection, rectangular}).  In Section \ref{sec: enumeration}, we focus upon the enumeration of these tableaux, yielding a family of integers $C_{3,n}^k$ that function as a three-dimensional analogue of the $k$-Catalan numbers.  Closed formulas for $C_{3,n}^k$ are derived for $n \leq 5$ (Propositions \ref{thm: n=2}, \ref{thm: n=3}, \ref{thm: n=4}, \ref{thm: n=5}) and a general calculus is introduced to tackle the general case (Proposition \ref{thm: general n recurrences}).  In Section \ref{sec: properties}, we further explore the bijection between our standard set-valued Young tableaux and $k$-ary product-coproduct prographs.  Appropriately generalized prographs are placed in bijection with various sets of (non-rectangular and skew) standard set-valued Young tableaux (Theorem \ref{thm: prograph vs tableaux bijection, non-rectangular}), and a 180-degree rotation of $k$-ary prographs is shown to correspond to a set-valued analogue of the Sch\"utzenberger involution on standard Young tableaux (Theorem \ref{thm: schutzenberger}).  Section \ref{sec: future directions} closes the paper with a series of more cursory discussions, including a suggestion of additional combinatorial interpretations for $C_{d,n}^k$ and a consideration of three- and four-row set-valued tableaux with densities other than $\rho=(1,k-1,1)$.

\section{$k$-ary Product-Coproduct Prographs and Set-Valued Tableaux}
\label{sec: prographs}

Let $\mathcal{T}_n^k$ denote the set of full $k$-ary trees with $kn+1$ vertices, drawn so that the root vertex lies at the bottom of the tree.  It is well-known that $\mathcal{T}_n^k$ is enumerated by the $k$-Catalan number $C_n^k = \frac{(kn)!}{(kn-n+1)!n!}$.  In the case of $k=2$, this prompts the well-studied bijection between $\mathcal{T}_n^2$ and $S(n^2)$ that associates entries in the top row of the tableau to left children and entries in the bottom row of the tableau to right children.

The bijection between $\mathcal{T}_n^2$ and $S(n^2)$ may be generalized to a bijection between $\mathcal{T}_n^k$ and standard set-valued Young tableaux $\svt(n^2,\rho)$ with row-constant density $\rho = (k-1,1)$.  As described by Heubach, Li and Mansour \cite{HLM}, this generalized bijection $\phi_k: \mathcal{T}_n^k \rightarrow \svt(n^2,\rho)$ is defined as below.  For an example of $\phi_k$, see Figure \ref{fig: k-ary tree vs tableau}.

\begin{enumerate}
\item For any $T \in \mathcal{T}_n^k$, label the edges of $T$ with the integers $\lbrace 1,\hdots,nk \rbrace$ according to a depth-left first search.
\item Place all integers that label rightmost-children of $T$ in the bottom row of $\phi(T) \in \svt(n^2,\rho)$, in increasing order from left to right.
\item Place all remaining integers from $\lbrace 1, \hdots, nk \rbrace$ in the top row of $\phi(T)$, in increasing order from left to right and ensuring that each cell in the top row receives precisely $k-1$ integers.
\end{enumerate}

\begin{figure}[ht!]
\centering
\begin{tikzpicture}
[scale=0.6,auto=left,every node/.style={circle,fill=black,inner sep=1.15pt}]
	\node[draw] (b) at (0,0) {};
	\node[draw] (1l) at (-1.5,1.5) {};
	\node[draw] (1m) at (0,1.5) {};
	\node[draw] (1r) at (1.5,1.5) {};
	\node[draw] (2l) at (-2.5,3) {};
	\node[draw] (2m) at (-1.5,3) {};
	\node[draw] (2r) at (-0.5,3) {};
	\node[draw] (3l) at (0.5,3) {};
	\node[draw] (3m) at (1.5,3) {};
	\node[draw] (3r) at (2.5,3) {};	
	\draw[thick] (b) to (1l);
	\draw[thick] (b) to (1m);
	\draw[thick] (b) to (1r);
	\draw[thick] (1l) to (2l);
	\draw[thick] (1l) to (2m);
	\draw[thick] (1l) to (2r);
	\draw[thick] (1r) to (3l);
	\draw[thick] (1r) to (3m);
	\draw[thick] (1r) to (3r);
	\node[fill=none] (1) at (-1.0,0.5) {1};
	\node[fill=none] (2) at (-2.3,2.1) {2};	
	\node[fill=none] (3) at (-1.8,2.6) {3};
	\node[fill=none] (4) at (-1.0,2.75) {4};
	\node[fill=none] (5) at (-0.25,1.0) {5};
	\node[fill=none] (6) at (0.65,1.1) {6};	
	\node[fill=none] (7) at (0.6,2.2) {7};
	\node[fill=none] (8) at (1.25,2.6) {8};
	\node[fill=none] (9) at (2,2.7) {9};	
\end{tikzpicture}
\hspace{.2in}
\raisebox{21pt}{\scalebox{2.5}{$\Leftrightarrow$}}
\hspace{.2in}
\raisebox{29pt}{\scalebox{1}{\begin{ytableau}
1 \kern+2pt 2 & 3 \kern+2pt 5 & 7 \kern+2pt 9 \\
4 & 6 & 9
\end{ytableau}}}
\caption{An example of the bijection $\phi_k: \mathcal{T}_n^k \rightarrow \svt(n^2,\rho)$ for $k=3$.}
\label{fig: k-ary tree vs tableau}
\end{figure}
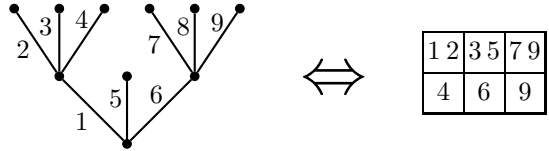

Following Borie \cite{Borie}, generalizing $\phi_k$ to three-row tableaux requires a consideration of prographs.  For any finite collection $G$ of formal operators, each of which is uniquely identified by its number of inputs and outputs, one may consider the set of all finite compositions that are freely constructed using elements of $G$ (as well as the identity operator $\id$).  Each of these compositions corresponds to a directed planted plane graph in which all edges are directed upward.  In these graphs, each application of a non-identity operator corresponds to a non-initial, non-terminal vertex whose vertical placement (when read from bottom to top) corresponds to the stage at which the operator appears in the composition.  See Figure \ref{fig: basic prograph example} for a quick example.  If one enforces a notion of equivalence for planted plane graphs with the added condition that all edges must maintain a strictly upward orientation, the resulting set $\pro_G$ is referred to as the (free) prographs generated by $G$.

Let $A$ denote a formal module.  If $G$ consists solely of an operator $\Delta_k : A \rightarrow A \otimes \cdots \otimes A$ with $1$ input and $k$ outputs (a non-coassociative $k$-ary coproduct), elements of $\pro_G$ are a directed variation of full $k$-ary trees where every edge has been directed upward and a single input edge has been added below the root vertex.  The subset of these prographs with precisely $n$ usages of $\Delta_k$ are in bijection $\mathcal{T}_n^k$.  For example, the $3$-ary tree on the left side of Figure \ref{fig: k-ary tree vs tableau} corresponds the prograph shown in Figure \ref{fig: basic prograph example}.

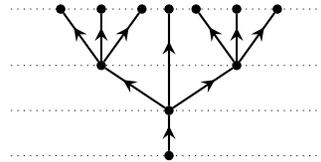
\begin{figure}[ht!]
\centering
\begin{tikzpicture}
[scale=0.6,auto=left,every node/.style={circle,fill=black,inner sep=1.1pt}]
	\node[draw] (bb) at (0,0) {};
	\node[draw] (b) at (0,1) {};
	\node[draw] (1l) at (-1.5,2) {};
	\node[draw] (1m) at (0,3.25) {};
	\node[draw] (1r) at (1.5,2) {};
	\node[draw] (2l) at (-2.4,3.25) {};
	\node[draw] (2m) at (-1.5,3.25) {};
	\node[draw] (2r) at (-0.6,3.25) {};
	\node[draw] (3l) at (0.6,3.25) {};
	\node[draw] (3m) at (1.5,3.25) {};
	\node[draw] (3r) at (2.4,3.25) {};	
	\draw[thick,directed] (bb) to (b);
	\draw[thick,directed] (b) to (1l);
	\draw[thick,directed] (b) to (1m);
	\draw[thick,directed] (b) to (1r);
	\draw[thick,directed] (1l) to (2l);
	\draw[thick,directed] (1l) to (2m);
	\draw[thick,directed] (1l) to (2r);
	\draw[thick,directed] (1r) to (3l);
	\draw[thick,directed] (1r) to (3m);
	\draw[thick,directed] (1r) to (3r);
	\draw[dotted] (-3.5,0) to (3.5,0) {};
	\draw[dotted] (-3.5,1) to (3.5,1) {};
	\draw[dotted] (-3.5,2) to (3.5,2) {};
	\draw[dotted] (-3.5,3.25) to (3.5,3.25) {};
\end{tikzpicture}
\caption{The prograph corresponding to the composition $(\Delta_3 \otimes \id \otimes \Delta_3) \circ \Delta_3$, where $\Delta_3$ is a $3$-ary coproduct.}
\label{fig: basic prograph example}
\end{figure}

Now consider the case where $G$ consists of a (non-coassociative) $k$-ary coproduct $\Delta_k: A \rightarrow A \otimes \cdots \otimes A$ with $1$ input and $k$ outputs, as well as a (non-associative) $k$-ary product $\mu_k: A \otimes \cdots \otimes A \rightarrow A$ with $k$ inputs and $1$ output. We refer to the resulting elements of $\pro_G$ as \textbf{$\mathbf{k}$-ary (product-coproduct) prographs}.  The subset of these prographs that have a single terminal vertex are known as \textbf{closed $\mathbf{k}$-ary (product-coproduct) prographs}.  As all prographs have a single initial vertex, all closed $k$-ary prographs must feature the same number of product and coproduct nodes.  We denote the set of all closed $k$-ary prographs with precisely $n$ products and $n$ product by $\pc^k(n)$.  See Figure \ref{fig: closed prograph example} for an illustration of $\pc^2(2)$.  Included in that figure is a representive from the equivalence class of compositions to which each prograph corresponds.

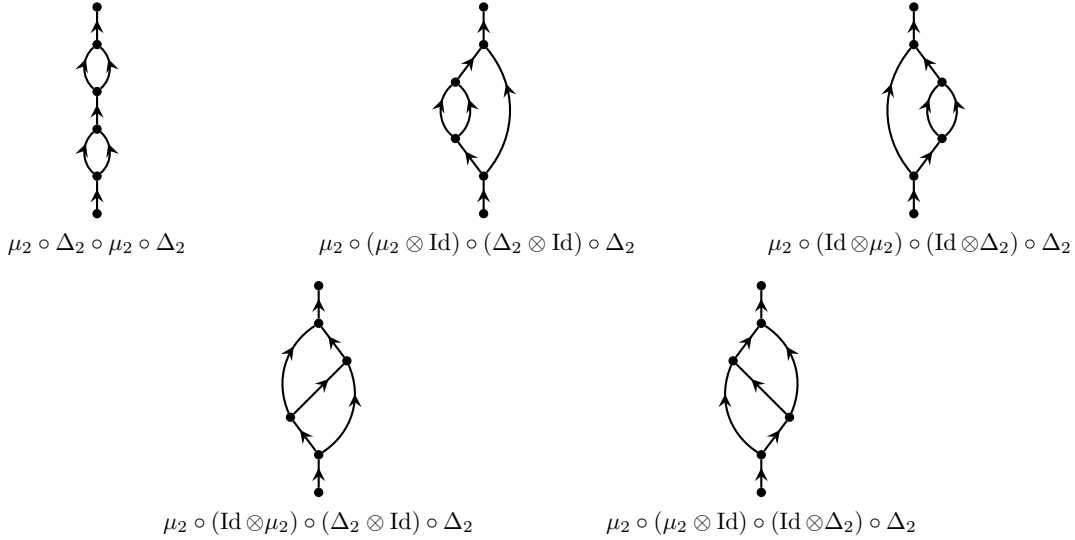
\begin{figure}[ht!]
\centering
\setlength{\tabcolsep}{25pt}
\begin{tabular}{c c c}
\begin{tikzpicture}
[scale=0.50,auto=left,every node/.style={circle,fill=black,inner sep=1.1pt}]
	\node[draw] (base) at (0,0) {};
	\node[draw] (a) at (0,1) {};
	\node[draw] (b) at (0,2.25) {};	
	\node[draw] (d) at (0,3.25) {};
	\node[draw] (e) at (0,4.5) {};
	\node[draw] (f) at (0,5.5) {};
	\draw[thick,directed] (base) to (a);
	\draw[thick,directed,bend left=50] (a) to (b);
	\draw[thick,directed,bend right=50] (a) to (b);
	\draw[thick,directed] (b) to (d);
    \draw[thick,directed,bend left=50] (d) to (e);
	\draw[thick,directed,bend right=50] (d) to (e);
	\draw[thick,directed] (e) to (f);
\end{tikzpicture}
&
\begin{tikzpicture}
[scale=0.50,auto=left,every node/.style={circle,fill=black,inner sep=1.1pt}]
	\node[draw] (base) at (0,0) {};
	\node[draw] (a) at (0,1) {};
	\node[draw] (b) at (-0.75,2) {};	
	\node[draw] (d) at (-0.75,3.5) {};
	\node[draw] (e) at (0,4.5) {};
	\node[draw] (f) at (0,5.5) {};
	\draw[thick,directed] (base) to (a);
	\draw[thick,directed] (a) to (b);
	\draw[thick,directed,bend left=50] (b) to (d);
	\draw[thick,directed,bend right=50] (b) to (d);
	\draw[thick,directed] (d) to (e);
	\draw[thick,directed,bend right=40] (a) to (e);
	\draw[thick,directed] (e) to (f);
\end{tikzpicture}
&
\begin{tikzpicture}
[scale=0.50,auto=left,every node/.style={circle,fill=black,inner sep=1.1pt}]
	\node[draw] (base) at (0,0) {};
	\node[draw] (a) at (0,1) {};
	\node[draw] (b) at (0.75,2) {};	
	\node[draw] (d) at (0.75,3.5) {};
	\node[draw] (e) at (0,4.5) {};
	\node[draw] (f) at (0,5.5) {};
	\draw[thick,directed] (base) to (a);
	\draw[thick,directed] (a) to (b);
	\draw[thick,directed,bend left=50] (b) to (d);
	\draw[thick,directed,bend right=50] (b) to (d);
	\draw[thick,directed] (d) to (e);
	\draw[thick,directed,bend left=40] (a) to (e);
	\draw[thick,directed] (e) to (f);
\end{tikzpicture}
\\
\scalebox{.9}{$\mu_2 \circ \Delta_2 \circ \mu_2 \circ \Delta_2$}
&
\scalebox{.9}{$\mu_2 \circ (\mu_2 \otimes \id) \circ (\Delta_2 \otimes \id) \circ \Delta_2$}
&
\scalebox{.9}{$\mu_2 \circ (\id \otimes \mu_2) \circ (\id \otimes \Delta_2) \circ \Delta_2$}
\end{tabular}

\vspace{.1in}

\begin{tabular}{c c}
\begin{tikzpicture}
[scale=0.50,auto=left,every node/.style={circle,fill=black,inner sep=1.1pt}]
	\node[draw] (base) at (0,0) {};
	\node[draw] (a) at (0,1) {};
	\node[draw] (b) at (-0.75,2) {};	
	\node[draw] (d) at (0.75,3.5) {};
	\node[draw] (e) at (0,4.5) {};
	\node[draw] (f) at (0,5.5) {};
	\draw[thick,directed] (base) to (a);
	\draw[thick,directed] (a) to (b);
	\draw[thick,directed] (b) to (d);
	\draw[thick,directed,bend left=40] (b) to (e);
	\draw[thick,directed,bend right=40] (a) to (d);
	\draw[thick,directed] (d) to (e);
	\draw[thick,directed] (e) to (f);
\end{tikzpicture}
&
\begin{tikzpicture}
[scale=0.50,auto=left,every node/.style={circle,fill=black,inner sep=1.1pt}]
	\node[draw] (base) at (0,0) {};
	\node[draw] (a) at (0,1) {};
	\node[draw] (b) at (0.75,2) {};	
	\node[draw] (d) at (-0.75,3.5) {};
	\node[draw] (e) at (0,4.5) {};
	\node[draw] (f) at (0,5.5) {};
	\draw[thick,directed] (base) to (a);
	\draw[thick,directed] (a) to (b);
	\draw[thick,directed] (b) to (d);
	\draw[thick,directed,bend right=40] (b) to (e);
	\draw[thick,directed,bend left=40] (a) to (d);
	\draw[thick,directed] (d) to (e);
	\draw[thick,directed] (e) to (f);
\end{tikzpicture}
\\
\scalebox{.9}{$\mu_2 \circ (\id \otimes \mu_2) \circ (\Delta_2 \otimes \id) \circ \Delta_2$}
&
\scalebox{.9}{$\mu_2 \circ (\mu_2 \otimes \id) \circ (\id \otimes \Delta_2) \circ \Delta_2$}
\end{tabular}
\caption{The set $\pc^2(2)$ of closed $2$-ary prographs with $2$ products and $2$ coproducts.}
\label{fig: closed prograph example}
\end{figure}

Borie \cite{Borie} argued that $\pc^2(n)$ is enumerated by the three-dimensional Catalan number $C_{3,n} = \frac{2 (3n)!}{n!(n+1)!(n+2)!}$.  This is accomplished by placing $\pc^2(n)$ in bijection with the three-row standard Young tableaux $S(n^3)$.

Our goal for the rest of this section is to generalize Borie's bijection to $\pc^k(n)$ for all $k \geq 2$, where the appropriate $k$-generalization of $S(n^3)$ is standard set-valued Young tableaux $\svt(n^3,\rho)$ with row-constant density $\rho=(1,k-1,1)$.  We begin by introducing an algorithm for labelling the edges of any $G \in \pc^k(n)$ that generalizes both the depth-left first labelling of $k$-ary trees and Borie's depth-left search for elements of $\pc^2(n)$.

\begin{enumerate}
\item Take any $G \in \pc^k(n)$, and begin by labelling the sole output of the initial node of $G$ with the integer $0$.
\item For each $0 \leq i < nk$, recursively define a subgraph $G_i$ of $G$ consisting solely of edges labelled by $\lbrace 1,\hdots,i \rbrace$.  Then let $V_i$ denote the subset of nodes from $G$ such that every input to that node lies in $G_i$ and at least one output from that node lies in $G - G_i$.
\item Identify the highest labelled edge from $G_i$ that terminates at an element $v$ of $V_i$ (this needn't be the edge labelled $i$).  Then label the leftmost unlabelled edge of that vertex $v$ with $i+1$ and return to Step \#2.
\end{enumerate}

See Figure \ref{fig: left-ascending search example} for an example of this procedure, which we henceforth refer to as our (generalized) depth-left first search.  Colloquially, the procedure may be described as ``staying as leftward as possible, with the restriction that all inputs to a node must be labelled before any output from that node may be labelled".  Also notice that this procedure directly generalizes to non-closed $k$-ary prographs: one merely needs to omit terminal nodes from the $V_i$ and repeat the recursive part of the algorithm until all terminal edges are labelled.

\begin{figure}[ht!]
\centering
\scalebox{0.85}{
\begin{tikzpicture}
[scale=.45,auto=left,every node/.style={circle,fill=black,inner sep=1.15pt}]
    \node[draw] (a) at (0,-1.5) {};
    \node[draw] (b) at (0,0) {};
    \node[draw] (c) at (1.5,1.5) {};
    \node[draw] (d) at (-1.5,1.5) {};
    \node[draw] (e) at (0,3) {};
    \node[draw] (f) at (0.75,4.5) {};
    \node[draw] (g) at (-0.75,6) {};
    \node[draw] (h) at (-0.75,7.5) {};
    \draw[thick,directed] (a) to (b);
    \draw[thick,directed] (b) to (c);
    \draw[thick,directed] (b) to (d);
    \draw[thick,directed] (c) to (e);
    \draw[thick,directed] (d) to (e);
    \draw[thick,directed,bend left=20] (d) to (g);
    \draw[thick,directed] (f) to (g);
    \draw[thick,directed,bend right=20] (c) to (f);
    \draw[thick,directed] (g) to (h);
    \draw[thick,directed,bend left=10] (e) to (f);
    \node[fill=none] (0) at (-0.4,-0.8) {0};
\end{tikzpicture}
\hspace{0.2in}
\raisebox{48pt}{\scalebox{3}{$\Rightarrow$}}
\hspace{0.1in}
\begin{tikzpicture}
[scale=.45,auto=left,every node/.style={circle,fill=black,inner sep=1.15pt}]
    \node[draw] (a) at (0,-1.5) {};
    \node[draw] (b) at (0,0) {};
    \node[draw] (c) at (1.5,1.5) {};
    \node[draw] (d) at (-1.5,1.5) {};
    \node[draw] (e) at (0,3) {};
    \node[draw] (f) at (0.75,4.5) {};
    \node[draw] (g) at (-0.75,6) {};
    \node[draw] (h) at (-0.75,7.5) {};
    \draw[thick,directed] (a) to (b);
    \draw[thick,directed] (b) to (c);
    \draw[thick,directed] (b) to (d);
    \draw[thick,directed] (c) to (e);
    \draw[thick,directed] (d) to (e);
    \draw[thick,directed,bend left=20] (d) to (g);
    \draw[thick,directed] (f) to (g);
    \draw[thick,directed,bend right=20] (c) to (f);
    \draw[thick,directed] (g) to (h);
    \draw[thick,directed,bend left=10] (e) to (f);
    \node[fill=none] (0) at (-0.4,-0.8) {0};
    \node[fill=none] (1) at (-1.15,0.5) {1};
    \node[fill=none] (2) at (-2.05,3.75) {2};
    \node[fill=none] (3) at (-0.95,2.7) {3};
\end{tikzpicture}
\hspace{0.2in}
\raisebox{48pt}{\scalebox{3}{$\Rightarrow$}}
\hspace{0.1in}
\begin{tikzpicture}
[scale=.45,auto=left,every node/.style={circle,fill=black,inner sep=1.15pt}]
    \node[draw] (a) at (0,-1.5) {};
    \node[draw] (b) at (0,0) {};
    \node[draw] (c) at (1.5,1.5) {};
    \node[draw] (d) at (-1.5,1.5) {};
    \node[draw] (e) at (0,3) {};
    \node[draw] (f) at (0.75,4.5) {};
    \node[draw] (g) at (-0.75,6) {};
    \node[draw] (h) at (-0.75,7.5) {};
    \draw[thick,directed] (a) to (b);
    \draw[thick,directed] (b) to (c);
    \draw[thick,directed] (b) to (d);
    \draw[thick,directed] (c) to (e);
    \draw[thick,directed] (d) to (e);
    \draw[thick,directed,bend left=20] (d) to (g);
    \draw[thick,directed] (f) to (g);
    \draw[thick,directed,bend right=20] (c) to (f);
    \draw[thick,directed] (g) to (h);
    \draw[thick,directed,bend left=10] (e) to (f);
    \node[fill=none] (0) at (-0.4,-0.8) {0};
    \node[fill=none] (1) at (-1.15,0.5) {1};
    \node[fill=none] (2) at (-2.05,3.75) {2};
    \node[fill=none] (3) at (-0.95,2.7) {3};
    \node[fill=none] (4) at (1.05,0.45) {4};
    \node[fill=none] (5) at (0.9,2.7) {5};
    \node[fill=none] (6) at (-0.1,4) {6};
\end{tikzpicture}
\hspace{0.2in}
\raisebox{48pt}{\scalebox{3}{$\Rightarrow$}}
\hspace{0.1in}
\begin{tikzpicture}
[scale=.45,auto=left,every node/.style={circle,fill=black,inner sep=1.15pt}]
    \node[draw] (a) at (0,-1.5) {};
    \node[draw] (b) at (0,0) {};
    \node[draw] (c) at (1.5,1.5) {};
    \node[draw] (d) at (-1.5,1.5) {};
    \node[draw] (e) at (0,3) {};
    \node[draw] (f) at (0.75,4.5) {};
    \node[draw] (g) at (-0.75,6) {};
    \node[draw] (h) at (-0.75,7.5) {};
    \draw[thick,directed] (a) to (b);
    \draw[thick,directed] (b) to (c);
    \draw[thick,directed] (b) to (d);
    \draw[thick,directed] (c) to (e);
    \draw[thick,directed] (d) to (e);
    \draw[thick,directed,bend left=20] (d) to (g);
    \draw[thick,directed] (f) to (g);
    \draw[thick,directed,bend right=20] (c) to (f);
    \draw[thick,directed] (g) to (h);
    \draw[thick,directed,bend left=10] (e) to (f);
    \node[fill=none] (0) at (-0.4,-0.8) {0};
    \node[fill=none] (1) at (-1.1,0.5) {1};
    \node[fill=none] (2) at (-2.05,3.75) {2};
    \node[fill=none] (3) at (-0.95,2.7) {3};
    \node[fill=none] (4) at (1.05,0.45) {4};
    \node[fill=none] (5) at (0.9,2.7) {5};
    \node[fill=none] (6) at (-0.1,4) {6};
    \node[fill=none] (7) at (1.8,3.25) {7};
    \node[fill=none] (8) at (0.3,5.6) {8};
    \node[fill=none] (9) at (-1.15,6.75) {9};
\end{tikzpicture}}
\caption{An example of our generalized depth-left first search, applied to an element of $\pc^2(3)$.}
\label{fig: left-ascending search example}
\end{figure}
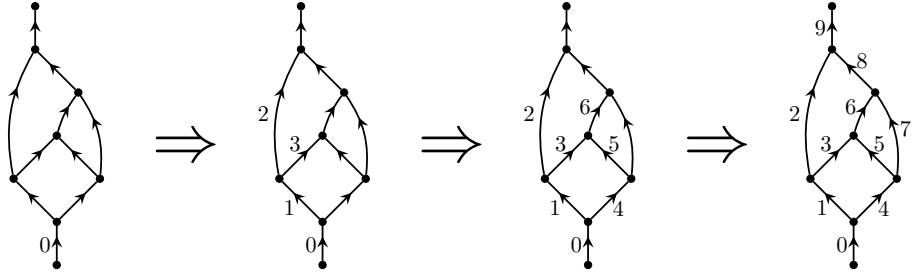

We are now ready to place $\pc^k(n)$ in bijection with an appropriate collection of standard set-valued Young tableaux.  Observe that Theorem \ref{thm: prograph vs tableaux bijection, rectangular} directly recovers the result of Borie \cite{Borie} in the case of $k=2$.

\begin{theorem}
Fix $n \geq 1$ and $k \geq 2$.  Then $\vert \pc^k(n) \vert = \vert \svt(\lambda,\rho) \vert$ for $\lambda = n^3$ and $\rho=(1,k-1,1)$.
\label{thm: prograph vs tableaux bijection, rectangular}
\end{theorem}
\begin{proof}
We provide a pair of well-defined functions $\Phi : \pc^k(n) \rightarrow \mathbb{S}(n^3,\rho)$, $\Phi_2: \mathbb{S}(n^3,\rho) \rightarrow \pc^k(n)$ and then show that $\Phi_2 = \Phi^{-1}$.  Our first map $\Phi: \pc^k(n) \rightarrow \svt(n^3,\rho)$ is defined as below.  See Figure \ref{fig: prograph vs tableaux example, rectangular} for an example.
\begin{enumerate}
\item For any $G \in \pc^k(n)$, label the edges of $G$ according to our depth-left first search.
\item Place integers that label leftmost coproduct children of $G$ in the top row of $\Phi(G) \in \svt(n^3,\rho)$, in increasing order from left to right.
\item Place integers that label all remaining coproduct children of $G$ along the middle row of $\Phi(G)$, in increasing order from left to right and ensuring that each cell in the middle row receives precisely $k-1$ integers.
\item Place integers corresponding to product children of $G$ along the bottom row of $\Phi(G)$, in increasing order from left to right.
\end{enumerate}

Notice that the initial input label of $0$ is ignored in this procedure.  As $\Phi(G)$ is row-standard by construction, to show that $\Phi$ is a well-defined map into $\svt(n^3,\rho)$ we merely need to argue that $\Phi(G)$ is column-standard.  Begin by noticing that our depth-left first search ensures that the leftmost child of a given coproduct node will always be labelled prior to the $k-1$ non-leftmost children of that same node.  This implies that every entry in the middle row of $\Phi(G)$ must be larger than the entry in the top row of the same column.

Now assume that precisely $\alpha_1$ leftmost coproduct children and $\alpha_2$ other coproduct children have been labelled prior to the labelling of the $j^{th}$ product child of $G$.  In order for the $j^{th}$ product child to receive the next label, there must have been at least $k$ previously labelled edges terminating at a node with an unlabelled output.  Among the $\alpha_1 + \alpha_2 + (j-1) + 1$ edges that were labelled prior to the labelling of the $j^{th}$ product child (initial $0$ edge included), precisely $k$ edges terminate at each of the $j-1$ product nodes with a previously labelled output, while $1$ edge terminates at each of the $\alpha_1$ coproduct nodes with a previously labelled leftmost output.  This leaves $\alpha_1 + \alpha_2 + (j-1) + 1 - k(j-1) - \alpha_1 = \alpha_2 - (k-1)j + k$ labelled edges that could lead into a product node with an unlabelled output.  Enforcing $\alpha_2 - (k-1)j + k \geq k$ gives $\alpha_2 \geq (k-1)j$, ensuring that all entries in the middle row of the $j^{th}$ column are smaller than the entry in the bottom row of the $j^{th}$ column.  It follows that $\Phi(G)$ is in fact column-standard and hence that $\Phi$ is well-defined.

For our second map $\Phi_2: \svt(n^3,\rho) \rightarrow \pc^k(n)$, we recursively ``build up" an edge-labelled prograph by working through $T \in \svt(n^3,\rho)$ one entry at a time, as described below.

\begin{enumerate}
\item For any $T \in \svt(n^3,\rho)$, begin by placing an initial input edge labelled $0$.  Then recursively consider each entry $1 \leq i \leq (k+1)n$ in numerical order.
\item If $i$ lies in the top row of $T$, place a coproduct node whose input is the edge labelled $i-1$.  Then label the leftmost child of that coproduct with $i$.
\item If $i$ is in the middle row of $T$, follow the depth-left first search through the partially constructed graph from the edge labelled $i-1$.  Then label the first unlabelled edge you encounter with the integer $i$.
\item If $i$ is in the bottom row of $T$, place a product node whose rightmost input is the edge labelled $i-1$ and whose remaining inputs are the $k-1$ nearest terminal edges immediately to the left of the edge labelled $i-1$.  Then label the output of that product $i$.
\end{enumerate}

The well-definedness of $\Phi_2$ depends upon whether the actions described above are possible at every step.  In particular, there must exist a rightward unlabelled edge when applying Step \#3, and there must be enough leftward free edges (all previously labelled) when adding the product node in Step \#4.

Begin by noting that, in the procedure that constructs $\Phi_2(T)$, leftmost coproduct children and product children are labelled as soon as they are placed.  This means that unlabelled terminal edges at any intermediate step must correspond to non-leftmost coproduct children, and hence that all edges labelled in Step \#3 must be non-leftmost coproduct children.  Also notice that the edge labelled $i$ immediately serves as an input for a new product or coproduct node unless $i+1$ lies in the middle row of $T$.  As this case involves an application of the depth-left first search, when it is initially placed $i+1$ is always the rightmost terminal edge in our partially constructed prograph.

So assume that the entry $i$ lies in the cell $(2,j)$ of $T$, and that $i$ is larger than precisely $x$ other integers in that cell ($0 \leq x \leq k-2$).  Row- and column-standardness of $T$ guarantees that at least $j$ coproducts have already been placed prior to this step, and that $(k-1)(j-1) + x$ of the non-leftmost children from those coproducts have already been labelled.  This means there are at least $(k-1)j - (k-1)(j-1) - x = k-1- x \geq 1$ unlabelled non-leftmost coproduct children at this step.  Because all non-leftmost coproduct children are labelled according to our depth-left first search, all of these unlabelled coproduct edges lie to the left of the edge labelled $i-1$.  Thus the operation of Step \#3 is always possible.

Now assume that $i$ lies in the cell $(3,j)$ of $T$.  Row- and column-standardness of $T$ guarantee that at least $j$ coproducts and precisely $j-1$ products have already been placed at this point in the procedure, with at least $kj$ coproduct children and precisely $j-1$ product children having been labelled.  As $j-1$ labelled inputs are needed for the placement of each coproduct, this means that there are at least $kj - (k-1)(j-1) - (j-1) = k$ labelled free edges when $i$ is the active integer.  Via preceding comments, the edge labelled $i-1$ is the rightmost of these free edges.  Thus the operation of Step \#4 is always possible, and we may conclude that $\Phi_2$ is well-defined.

It is only left to show that $\Phi_2 = \Phi^{-1}$.  We demonstrate that $\Phi_2 \circ \Phi (G) = G$ for any $G \in \pc^k(n)$, and that $\Phi \circ \Phi_2(T)$ for any $T \in \svt(n^3,\rho)$.

To show $\Phi_2 \circ \Phi(G) = G$ for any $G \in \pc^k(n)$, we inductively work through the edges of $G$ in the order of the depth-left first search.  For $i=0$, $G$ and $\Phi_2 \circ \Phi (G)$ both feature a single input edge labelled with $i$.  For any $1 \leq i \leq n(k+1)$, assume that $G$ and $\Phi_2 \circ \Phi(G)$ feature identical sub-prographs (not necessarily closed) corresponding to the edges labelled $\lbrace 0,\hdots,i-1 \rbrace$.  There are the three possible scenarios for the edge labelled $i$.

\begin{enumerate}
\item If $i-1$ labels the input to a coproduct node in $G$, the edge labelled $i$ must be the leftmost output of that same coproduct.  This implies that $i$ lies in the top row of $\Phi(G)$ and hence that $\Phi_2 \circ \Phi(G)$ also features a coproduct with input $i-1$ and leftmost output $i$.
\item If $i-1$ labels the rightmost input to a product node in $G$, the edge labelled $i$ in $G$ is always the next (on the right) input to that same product.  That means that $i$ lies in the middle row of $\Phi(G)$ and that the edge labelled with $i$ in $\Phi_2 \circ \Phi(G)$ is determined via a depth-left first search from the edge labelled $i-1$.  This results in the next (on the right) input to that same product being labelled $i$ in $\Phi_2 \circ \Phi(G)$.
\item If $i-1$ labels the rightmost input to a product node in $G$, the edge labelled $i$ in $G$ is necessarily the output of that product.  This implies that $i$ lies in the bottom row of $\Phi(G)$ and thus that $i$ also labels a product output in $\Phi_2 \circ \Phi(G)$ whose rightmost input is labelled $i-1$.
\end{enumerate}

As all three options lead to an identical placement of the edge labelled with $i$, we conclude $\Phi_2 \circ \Phi(G) = G$.

To show that $\Phi \circ \Phi_2 (T) = T$ for any $T \in \svt(n^3,\rho)$, we inductively work through the entries of $T$.  For $i=1$, $T$ and $\Phi \circ \Phi_2(T)$ both feature $i$ in the top-left corner.  For $2 \leq i \leq n(k+1)$, assume that $T$ and $\Phi \circ \Phi_2(T)$ feature identical subtableau corresponding to the entries $\lbrace 1,\hdots,i-1 \rbrace$.  There are once again three possibilities for $i$:

\begin{enumerate}
\item If $i$ lies in the top row of $T$, $i$ labels a leftmost child of a coproduct node in $\Phi_2(T)$ whose input is labelled $i-1$.  Thus $i$ lies in the top row of $\Phi \circ \Phi_2(T)$.
\item If $i$ lies in the middle row of $T$, via earlier comments we know that $i$ will always label a non-leftmost coproduct child in $\Phi_2(T)$.  It follows that $i$ also lies in the middle row of $\Phi \circ \Phi_2(T)$
\item If $i$ lies in the bottom row of $T$, $i$ labels a product output in $\Phi_2(T)$ and hence $i$ also lies in the bottom row of $\Phi \circ \Phi_2(T)$.
\end{enumerate}

As all three cases lead to identical placement of $i$ in the relevant tableaux, we conclude $\Phi \circ \Phi_2(T) = T$.
\end{proof}

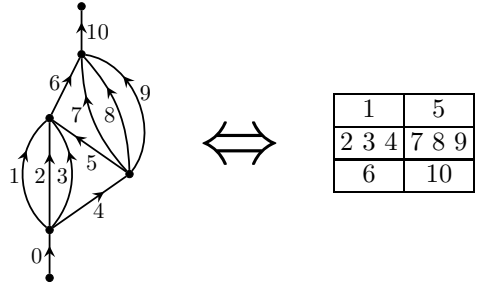
\begin{figure}[ht!]
\centering
\scalebox{.85}{\begin{tikzpicture}
[scale=1,auto=left,every node/.style={circle,fill=black,inner sep=1.15pt}]
    \node[draw] (a) at (0,.25) {};
    \node[draw] (b) at (0,1) {};
    \node[draw] (c) at (1.25,1+1.75/2) {};
    \node[draw] (d) at (0,2.75) {};
    \node[draw] (e) at (0.5,3.75) {};
    \node[draw] (f) at (0.5,4.5) {};
    \draw[thick,directed] (a) to (b);
    \draw[thick,directed] (b) to (c);
    \draw[thick,directed,bend right=40] (b) to (d);
    \draw[thick,directed] (b) to (d);
    \draw[thick,directed,bend left=50] (b) to (d);
    \draw[thick,directed] (c) to (d);
    \draw[thick,directed,bend right=60] (c) to (e);
    \draw[thick,directed,bend right=20] (c) to (e);
    \draw[thick,directed,bend left=20] (c) to (e);
    \draw[thick,directed] (d) to (e);
    \draw[thick,directed] (e) to (f);
    \node[fill=none] (0) at (-.2,.6) {0};
    \node[fill=none] (1) at (-.55,1.85) {1};
    \node[fill=none] (2) at (-0.15,1.85) {2};
    \node[fill=none] (3) at (.2,1.85) {3};
    \node[fill=none] (4) at (.77,1.3) {4};
    \node[fill=none] (5) at (0.66,2.05) {5};
    \node[fill=none] (6) at (.08,3.3) {6};
    \node[fill=none] (7) at (0.42,2.8) {7};
    \node[fill=none] (8) at (.94,2.8) {8};
    \node[fill=none] (9) at (1.5,3.15) {9};
    \node[fill=none] (10) at (.75,4.1) {10};
\end{tikzpicture}}
\hspace{0.1in}
\raisebox{45pt}{\scalebox{3}{$ \Leftrightarrow$}}
\hspace{0.1 in}
\raisebox{50pt}{
\setlength{\tabcolsep}{2.2pt}
\begin{tabular}{|c|c|}
	\hline
	1 & 5 \\ \hline
	2 3 4 & 7 8 9 \\ \hline
	6 & 10 \\ \hline
\end{tabular}}
\caption{An example of the bijection $\Phi: \pc^k(n) \rightarrow \svt(n^3,\rho)$ for $n=2$ and $k=4$.}
\label{fig: prograph vs tableaux example, rectangular}
\end{figure}

\section{Enumerating $\svt(n^3,\rho)$ for $\rho = (1,k-1,1)$}
\label{sec: enumeration}

Theorem \ref{thm: prograph vs tableaux bijection, rectangular} suggests that $\svt(n^3,\rho)$ with $\rho=(1,k-1,1)$ generalizes $S(n^3)$ in a manner similar to how $\svt(n^2,\rho')$ with $\rho'=(k-1,1)$ generalizes $S(n^2)$.  As the $\svt(n^2,\rho')$ are enumerated by the $k$-Catalan numbers $C^k_n$, we henceforth refer to the cardinalities $\vert \svt(n^3,\rho) \vert = C^k_{3,n}$ as the \textbf{three-dimensional $k$-Catalan numbers}.

The purpose of this section is to develop closed formulas for $C^k_{3,n}$.  Sadly, developing such a formula or deriving a multivariate generating function for arbitrary $n \geq 1$, $k \geq 1$ do not appear to be tractable problems.  As such, we restrict our attention to cases of small $n$.  See Table \ref{tab: 1,k-1,1} of Appendix \ref{sec: appendix} for a table of known values of $C^k_{3,n}$, which combines the explicit results of this section with computer calculations performed in Java.

In all that follows, notice that the ``degenerate" $k=1$ case corresponds to three-row tableaux with empty cells across their middle row.  This means that the $k=1$ enumerations reduce to pre-existing results about two-row tableaux: that $C_{3,n}^1 = \vert \svt(n^3,\rho) \vert = \vert S(n^2) \vert = C_n$ for all $n \geq 1$ with $\rho = (1,0,1)$.

For all of our enumerations we recursively place $\svt(\lambda,\rho)$ in bijection with a collection of sets $\bigcup \svt(\lambda_i,\rho)$ of strictly smaller shape yet equivalent density.  Our technique is similar to pre-existing proofs for non-set-valued tableaux where the sub-shapes $\lambda_i$ are determined via the removal of lower-right corners, corresponding to possible locations of the largest possible entry in a tableau of shape $\lambda$.  The difference here is that we never remove entries from a cell without eliminating all entries in that cell.  If the removed cell contains entries other than the largest entry in the tableau, this necessitates that we account for the ordering of those smaller entries relative to integers appearing elsewhere in the tableau.

Before proceeding, observe that $\vert S(\lambda,\rho) \vert$ is easily calculable whenever $\lambda = (n,1,\hdots,1)$ is ``hook-shaped".  In this case, one merely needs to count the ways of partitioning entries between the rightward and downward ``legs", giving an enumeration in terms of a single binomial coefficient $\vert S(\lambda,\rho) \vert = \binom{a}{b}$.  See Figure \ref{fig: hook-shaped enumeration} for examples.

For the rest of this section, an unfilled Young diagram of shape $\lambda$ is used to denote the cardinality $\vert \svt(\lambda,\rho) \vert$, assuming $\rho = (1,k-1,1)$.

\begin{figure}[ht!]
\centering
\ytableausetup{boxsize=1em}
$$\raisebox{8pt}{\ydiagram{2,1,1}} \ = \ \binom{k+1}{1} \hspace{.5in} \raisebox{8pt}{\ydiagram{4,1,1}} \ = \ \binom{k+3}{3}$$
\caption{Cardinalities $\vert S(\lambda,\rho) \vert$ for $\rho = (1,k-1,1)$ and several hook-shapes $\lambda$.  As $1$ must lie at position $(1,1)$, one merely needs to determine which of $\lbrace 2,3,\hdots \rbrace$ lie in the remaining cells of the top row.}
\label{fig: hook-shaped enumeration}
\end{figure}

\begin{proposition}
\label{thm: n=2}
Let $\rho = (1,k-1,1)$.  For any $k \geq 1$, $C_{3,2}^k = \vert \svt(2^3,\rho) \vert = k^2 + 1$.
\end{proposition}
\begin{proof}
As the largest entry of any $T \in \svt(2^3,\rho)$ must lie at $(3,2)$, we investigate the integers $a_1 < \hdots < a_{k-1}$ lying at $(2,2)$ in an arbitrary set-valued tableaux $T_1 \in \svt(\lambda_1,\rho)$ of shape $\lambda_1 = (2,2,1)$.  The only other entry in $T_1$ that may be larger than any of the $a_i$ is the entry $b$ at position $(3,1)$.  The subset of $\svt(\lambda_1,\rho)$ satisfying $b \leq a_i$ for all $i$ is then in bijection with $\svt(\lambda_2,\rho)$ for $\lambda_2 = (2,1,1)$.  If $b > a_1$, one must specify the ordering of $b$ relative to $a_2,\hdots,a_{k-1}$.  So assume that $j$ is the largest index such that $a_j < b$ (where $1 \leq j \leq k-1$).  Each choice of $j$ defines a subset of $\svt(\lambda_2,\rho)$ that is in bijection with $\svt(\lambda_3,\rho)$ for $\lambda_3 = (2,1)$, since for any choice of $j$ the $k$ largest entries of such a tableau $T_1 \in \svt(\lambda_1,\rho)$ is split between positions $(2,2)$ and $(3,1)$.  Combining these observations gives the string of equalities below.

\ytableausetup{boxsize=1em}
$$\raisebox{8pt}{\ydiagram{2,2,2}} \ = \ \raisebox{8pt}{\ydiagram{2,2,1}} \ = \ \raisebox{8pt}{\ydiagram{2,1,1}} \ + \ (k-1) \ \raisebox{8pt}{\ydiagram{2,1}} \ = \ \binom{k+1}{1} \ + \ (k-1) \binom{k}{1} \ = \ k^2 + 1$$
\end{proof}

\begin{proposition}
\label{thm: n=3}
Let $\rho = (1,k-1,1)$.  For any $k \geq 1$, $$C_{3,3}^k = \vert \svt(3^3,\rho) \vert = \frac{9k^4 - 2k^3 + 9k^2}{4} + 1$$
\end{proposition}
\begin{proof}
We begin by enumerating $\svt(\lambda',\rho)$ for $\lambda'=(3,2,1)$.  For arbitrary $T' \in \svt(\lambda',\rho)$, let $a_1 < \hdots < a_{k-1}$ denote the entries at $(2,2)$, $b$ denote the entry at $(3,1)$, and $c$ denote the entry at $(1,3)$.  Proceeding as in the proof of Proposition \ref{thm: n=2}, we subdivide $\svt(\lambda',\rho)$ based on the relationship of $b$ and $c$ to the $a_i$ and then delete all entries $x \geq a_1$ to place each subset in bijection with tableaux of some smaller shape.  The equalities below synopsize our results, with the first summand corresponding to $b,c < a_1$, the second summand corresponding to the $k-1$ placements of $b$ relative to $a_2 < \hdots < a_{k-1}$ when $b > a_1$ yet $c < a_1$, the third summand corresponding to the $k-1$ placements of $c$ relative to $a_2 < \hdots < a_{k-1}$ when $c > a_1$ yet $b < a_1$, and the fourth summand corresponding to the $\binom{k}{k-2,1,1}$ placements of $b,c$ relative to $a_2 < \hdots < a_{k-1}$ when $b,c > a_1$.

\ytableausetup{boxsize=1em}
$$\raisebox{8pt}{\ydiagram{3,2,1}} \ = \ \raisebox{8pt}{\ydiagram{3,1,1}} \ + \ \binom{k-1}{1} \ \raisebox{8pt}{\ydiagram{3,1}} \ + \ \binom{k-1}{1} \ \raisebox{8pt}{\ydiagram{2,1,1}} \ + \ \binom{k}{k-2,1,1} \ \raisebox{8pt}{\ydiagram{2,1}}$$

$$= \ \binom{k+2}{2} + \binom{k-1}{1} \binom{k+1}{2} + \binom{k-1}{1} \binom{k+1}{1} + \binom{k}{k-2,1,1} \binom{k}{1} \ = \ \frac{3k^3 + k^2 + 2k}{2}$$

For the full theorem, we once again proceed as in the proof to Proposition \ref{thm: n=2}.  After reducing to arbitrary $T \in \svt(\lambda,\rho)$ with $\lambda = (3,3,2)$, we divide $\svt(\lambda,\rho)$ into subsets depending upon how the entries $a_1 < \hdots < a_{k-1}$ at position $(2,3)$ relate to the entry $b_1$ at $(3,1)$ and the entry $b_2$ at $(3,2)$.  The three summands in the first line of the equalities below corresponds to the cases of $b_1 < b_2 < a_1$, $b_1 < a_1 < b_2$, and $a_1 < b_1 < b_2$, respectively.  In the second line of equalities, the first of those subsets is further subdivided based upon the relationship of the entry $c$ at $(1,3)$ to the entry $b_2$ at $(3,2)$, with the two new summands corresponding to $b_2 < c$ and $c < b_2$, respectively.  This leaves a sum of cardinalities $\vert \svt(\lambda_i,\rho) \vert$ that are computable via Proposition \ref{thm: n=2}, our informal lemma for shape $\lambda'=(3,2,1)$, and the result of Heubach, Li and Mansour \cite{HLM} giving $\vert \svt(n^2,\rho) \vert = C^k_n$.

\ytableausetup{boxsize=1em}
$$\raisebox{8pt}{\ydiagram{3,3,3}} \ = \ \raisebox{8pt}{\ydiagram{3,3,2}} \ = \ \raisebox{8pt}{\ydiagram{3,2,2}} \ + \ \binom{k-1}{1} \ \raisebox{8pt}{\ydiagram{3,2,1}} \ + \ \binom{k}{2} \ \raisebox{8pt}{\ydiagram{3,2}}$$

$$= \ \raisebox{8pt}{\ydiagram{2,2,2}} \ + \ \raisebox{8pt}{\ydiagram{3,2,1}} \ + \ \binom{k-1}{1} \ \raisebox{8pt}{\ydiagram{3,2,1}} \ + \ \binom{k}{2} \ \raisebox{8pt}{\ydiagram{3,2}}$$

$$= \ (k^2 + 1) + k \left( \frac{3k^3 + k^2 + 2k}{2} \right) + \binom{k}{2} C^k_3 \ = \ \frac{9k^4 - 2k^3 + 9k^2}{4} + 1$$
\end{proof}

The proofs of Propositions \ref{thm: n=2} and \ref{thm: n=3} suggest a general methodology for enumerating $\svt(n^3,\rho)$ that could be applied to all $n \geq 2$.  In particular, for any three-row shape our technique of removing every entry in a lower-right corner yields the recurrences of Proposition \ref{thm: general n recurrences}.

\begin{proposition}
\label{thm: general n recurrences}
Fix $k \geq 1$.  For $\rho = (1,k-1,1)$ and any three-row shape $\lambda = (a,b,c)$ with $a \leq b \leq c$,

$$|\svt((a,b,c),\rho) | =
\begin{cases}
\displaystyle{\sum_{\substack{0 \leq i \leq a-b,\\[1pt] 0 \leq j \leq c}} \binom{k-2+i+j}{k-2,i,j} \kern+2pt \vert \svt((a-i,b-1,c-j),\rho) \vert}, & \text{if $b > c$;}\\[22pt]
\displaystyle{\sum_{0 \leq i \leq a-b} \vert \svt((a-i,b,c-1),\rho) \vert}, & \text{if $b = c$.}
\end{cases}$$
\end{proposition}

Notice that, although we have utilized other results about hook-shaped tableaux and two-row tableaux to shorten our proofs in the $n=2,3$ cases, the two recurrences of Proposition \ref{thm: general n recurrences} are sufficient to reduce any $\vert \svt((a,b,c),\rho) \vert$ to a summation involving one-column shapes $\lambda_i$, where $\vert \svt(\lambda_i,\rho) \vert = 1$.  Considered as a function of $k$, we may then use Proposition \ref{thm: general n recurrences} to quickly draw several conclusions about $\vert \svt((a,b,c),\rho) \vert$:

\begin{corollary}
\label{thm: recurrence corollary}
Fix $\rho=(1,k-1,1)$, where $k$ is indeterminate, and let $\lambda = (a,b,c)$ satisfy both $b \geq 1$ and $a+c \geq 2$.  If $a=b=c$, then $\vert \svt((a,b,c),\rho) \vert$ is a polynomial in $k$ of degree $a+c-2$.  If $a>c$, then $\vert \svt((a,b,c),\rho) \vert$ is a polynomial in $k$ of degree $a+c-1$.
\end{corollary}
\begin{proof}
That $\vert \svt((a,b,c),\rho) \vert = p(k)$ is a polynomial in $k$ follows directly from the recursion of Proposition \ref{thm: recurrence corollary}.  To demonstrate the degree of $p(k)$, induct on $i=a+b+c$ for $i\geq 3$.  The base case of $i = 3$ follows from $\vert \svt((2,1,0), \rho) \vert = k$ and $\vert \svt((1,1,1),\rho) \vert = 1$.  For the inductive case, take $\vert \svt((a,b,c),\rho) \vert$ with $a+b+c = i+1$.  If $b>c$, the first case of Proposition \ref{thm: recurrence corollary} equates $\vert \svt((a,b,c),\rho) \vert$ with a sum of polynomials (all with positive leading coefficient) whose maximal degree summand(s) all have degree $a+c-1$.  If $b=c$, the second case of Proposition \ref{thm: recurrence corollary} equates $\vert \svt((a,b,c),\rho) \vert$ with a sum of polynomials whose sole maximal degree summand has degree $a+(c-1)-1$.
\end{proof}

In the case of $\lambda = n^3$, notice that Corollary \ref{thm: recurrence corollary} implies that $p(k) = \vert \svt(n^3,\rho) \vert$ has degree $2(n-1)$.  For several additional enumerations, Proposition \ref{thm: general n recurrences} may be applied with the aid of a computer algebra system to derive the following polynomials for the $k=4$ and $k=5$ cases.

\begin{proposition}
\label{thm: n=4}
Let $\rho = (1,k-1,1)$.  For any $k \geq 1$, $$C_{3,4}^k = \vert \svt(4^3,\rho) \vert = \frac{256 k^6 - 114 k^5 + 217 k^4 - 12 k^3 + 121 k^2}{36} + 1$$
\end{proposition}

\begin{proposition}
\label{thm: n=5}
Let $\rho = (1,k-1,1)$.  For any $k \geq 1$, $$C_{3,5}^k = \vert \svt(5^3,\rho) \vert = \frac{15625 k^8 - 10092 k^7 + 10258 k^6 - 72 k^5 + 5473 k^4 - 204 k^3 + 2628 k^2}{576} + 1$$
\end{proposition}

\section{Properties of $k$-ary Product-Coproduct Prographs}
\label{sec: properties}

In this section we prove a generalization of Theorem \ref{thm: prograph vs tableaux bijection, rectangular} that applies to non-closed $k$-ary prographs satisfying certain basic properties.  We then explore one significant application of our bijection that generalizes an unproven proposition of Borie \cite{Borie}, showing that 180-degree rotation of prographs corresponds to a set-valued analogue of the Sch\"utzenberger involution on standard Young tableaux.

\subsection{Non-Closed $k$-ary Prographs and Set-Valued Tableaux}
\label{subsec: non-closed prographs}

We begin by generalizing the set $\pro_G$ to finite compositions of formal operators where the initial input is the an $x$-fold tensor product $A \otimes \cdots \otimes A$ of the formal module $A$.  The resulting directed plane graphs resemble prographs over $G$ but now contain precisely $x$ input strands, aligned horizontally across the bottom of the graph.  Fixing $G$ and $m \geq 1$, we may enforce a notion of equivalence on the resulting set of directed plane graphs that is analogous to the equivalence relation on $\pro_G$ from Section \ref{sec: prographs}.  We refer to the resulting set of equivalence classes $\pro_{G,x}$ as the set of $x$-fold (free) prographs generated by $G$.

In the case where $G$ consists of a $k$-ary coproduct $\Delta_k$ and a $k$-ary product $\mu_k$, we refer to the elements of $\pro_{G,x}$ as \textbf{$\mathbf{x}$-fold $\mathbf{k}$-ary (product-coproduct) prographs}.  We denote the subset of $x$-fold $k$-ary prographs with precisely $n$ coproduct nodes, $m$ product nodes, and $x$ input strands by $\pc_x^k(n,m)$.  Notice that these three parameters are sufficient to determine the number of output strands in any $G \in \pc_x^k(n,m)$.  Explicitly,

\begin{proposition}
\label{thm: number of output strands}
Take any $G \in \pc_x^k(n,m)$.  Then $G$ has precisely $y = (n-m)(k-1)+x$ output strands.  In particular, $y \equiv x \mod(k-1)$.
\end{proposition}
\begin{proof}
Observe that each $k$-ary coproduct increases the number of free edges by $k-1$, while each $k$-ary product decreases the number of free edges by $k-1$.  If we begin with $x$ free edges, after $n$ coproducts and $m$ products we have $y = x + n(k-1) - m(k-1)$ outgoing free edges.
\end{proof}

For any $x \equiv 1 \mod(k-1)$, consider $\pc_x^k(n,m)$.  There exists an injection $j: \pc_x^k(n,m) \rightarrow \pc^k(n + \frac{x-1}{k-1})$ that is defined by recursively joining incoming strands with $k$-ary coproducts, from left to right in sets of $k$, while recursively joining outgoing strands with $k$-ary products, from right to left in sets of $k$.  See Figure \ref{fig: justified prographs} for an illustration.  For any $G \in \pc_x^k(n,m)$, we call the image $j(G) \in \pc^k(n + \frac{x-1}{k-1})$ the \textbf{justification} of $G$.  Assuming $x \equiv 1 \mod(k-1)$, justification suggests the generalization of Theorem \ref{thm: prograph vs tableaux bijection, rectangular} given by Theorem \ref{thm: prograph vs tableaux bijection, non-rectangular}.

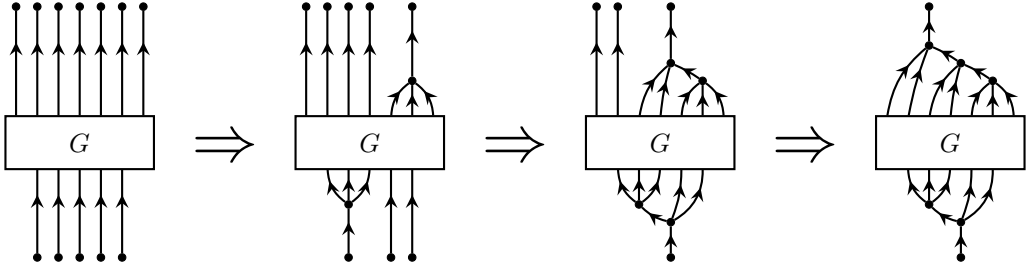
\begin{figure}[ht!]
\centering
\begin{tikzpicture}
[scale=.47,auto=left,every node/.style={circle,fill=black,inner sep=1.0pt}]
    \node[draw] (a1) at (-1.2,0) {};
    \node[draw] (a2) at (-.6,0) {};
    \node[draw] (a3) at (0,0) {};
    \node[draw] (a4) at (.6,0) {};
    \node[draw] (a5) at (1.2,0) {};
    \node[fill=none] (b1) at (-1.2,2.6) {};
    \node[fill=none] (b2) at (-.6,2.6) {};
    \node[fill=none] (b3) at (0,2.6) {};
    \node[fill=none] (b4) at (.6,2.6) {};
    \node[fill=none] (b5) at (1.2,2.6) {};
	\node[fill=none] (c0) at (-1.8,3.9) {}; 
    \node[fill=none] (c1) at (-1.2,3.9) {};
    \node[fill=none] (c2) at (-.6,3.9) {};
    \node[fill=none] (c3) at (0,3.9) {};
    \node[fill=none] (c4) at (.6,3.9) {};
    \node[fill=none] (c5) at (1.2,3.9) {};
   	\node[fill=none] (c6) at (1.8,3.9) {};
   	\node[draw] (d0) at (-1.8,7.1) {};  
    \node[draw] (d1) at (-1.2,7.1) {};
    \node[draw] (d2) at (-.6,7.1) {};
    \node[draw] (d3) at (0,7.1) {};
    \node[draw] (d4) at (.6,7.1) {};
    \node[draw] (d5) at (1.2,7.1) {};      
	\node[draw] (d6) at (1.8,7.1) {};
    \draw[thick] (-2.1,2.5) rectangle (2.1,4);
    \draw[thick,directed] (a1) to (b1) {};
    \draw[thick,directed] (a2) to (b2) {};
    \draw[thick,directed] (a3) to (b3) {};
    \draw[thick,directed] (a4) to (b4) {};
    \draw[thick,directed] (a5) to (b5) {};
    \draw[thick,directed] (c0) to (d0) {};    
    \draw[thick,directed] (c1) to (d1) {};
    \draw[thick,directed] (c2) to (d2) {};
    \draw[thick,directed] (c3) to (d3) {};
    \draw[thick,directed] (c4) to (d4) {};
    \draw[thick,directed] (c5) to (d5) {};    
    \draw[thick,directed] (c6) to (d6) {};
    \node[fill=none] at (0,3.25) {$G$};
\end{tikzpicture}
\hspace{.1in}
\raisebox{38pt}{\scalebox{2.5}{$\Rightarrow$}}
\hspace{.1in}
\begin{tikzpicture}
[scale=.47,auto=left,every node/.style={circle,fill=black,inner sep=1.0pt}]
    \node[draw] (a2) at (-.6,0) {};
    \node[draw] (a4) at (.6,0) {};
    \node[draw] (a5) at (1.2,0) {};
    \node[draw] (bottom1) at (-.6,1.5) {};
    \node[fill=none] (b1) at (-1.2,2.6) {};
    \node[fill=none] (b2) at (-.6,2.6) {};
    \node[fill=none] (b3) at (0,2.6) {};
    \node[fill=none] (b4) at (.6,2.6) {};
    \node[fill=none] (b5) at (1.2,2.6) {};
    \node[fill=none] (c0) at (-1.8,3.9) {};   
    \node[fill=none] (c1) at (-1.2,3.9) {};
    \node[fill=none] (c2) at (-.6,3.9) {};
    \node[fill=none] (c3) at (0,3.9) {};
    \node[fill=none] (c4) at (.6,3.9) {};
    \node[fill=none] (c5) at (1.2,3.9) {};
    \node[fill=none] (c6) at (1.8,3.9) {};
    \node[draw] (top1) at (1.2,5.0) {};
 	\node[draw] (d0) at (-1.8,7.1) {};    
    \node[draw] (d1) at (-1.2,7.1) {};
    \node[draw] (d2) at (-.6,7.1) {};
    \node[draw] (d3) at (0,7.1) {};
    \node[draw] (d5) at (1.2,7.1) {};
    \draw[thick] (-2.1,2.5) rectangle (2.1,4);
    \draw[thick,directed] (a2) to (bottom1) {};
    \draw[thick,directed,bend left=25] (bottom1) to (b1) {};
    \draw[thick,directed] (bottom1) to (b2) {};
    \draw[thick,directed,bend right=25] (bottom1) to (b3) {};
    \draw[thick,directed] (a4) to (b4) {};
    \draw[thick,directed] (a5) to (b5) {};
    \draw[thick,directed] (c0) to (d0) {};
    \draw[thick,directed] (c1) to (d1) {};
    \draw[thick,directed] (c2) to (d2) {};
    \draw[thick,directed] (c3) to (d3) {};
    \draw[thick,directed,bend left=25] (c4) to (top1) {};
    \draw[thick,directed] (c5) to (top1) {};
    \draw[thick,directed,bend right=25] (c6) to (top1) {};
    \draw[thick,directed] (top1) to (d5) {};  
    \node[fill=none] at (0,3.25) {$G$};
\end{tikzpicture}
\hspace{.1in}
\raisebox{38pt}{\scalebox{2.5}{$\Rightarrow$}}
\hspace{.1in}
\begin{tikzpicture}
[scale=.47,auto=left,every node/.style={circle,fill=black,inner sep=1.0pt}]
    \node[draw] (a4) at (.3,0) {};
    \node[draw] (bottom2) at (.3,1) {};
    \node[draw] (bottom1) at (-.6,1.5) {};
    \node[fill=none] (b1) at (-1.2,2.6) {};
    \node[fill=none] (b2) at (-.6,2.6) {};
    \node[fill=none] (b3) at (0,2.6) {};
    \node[fill=none] (b4) at (.6,2.6) {};
    \node[fill=none] (b5) at (1.2,2.6) {};
	\node[fill=none] (c0) at (-1.8,3.9) {}; 
    \node[fill=none] (c1) at (-1.2,3.9) {};
    \node[fill=none] (c2) at (-.6,3.9) {};
    \node[fill=none] (c3) at (0,3.9) {};
    \node[fill=none] (c4) at (.6,3.9) {};
    \node[fill=none] (c5) at (1.2,3.9) {};
    \node[fill=none] (c6) at (1.8,3.9) {};
    \node[draw] (top1) at (1.2,5.0) {};
    \node[draw] (top2) at (.3,5.5) {};
    \node[draw] (d0) at (-1.8,7.1) {};
    \node[draw] (d1) at (-1.2,7.1) {};
    \node[draw] (d4) at (.3,7.1) {};
    \draw[thick] (-2.1,2.5) rectangle (2.1,4);
    \draw[thick,directed] (a4) to (bottom2) {};
    \draw[thick,directed,bend left=15] (bottom2) to (bottom1) {};
    \draw[thick,directed,bend right=10] (bottom2) to (b4) {};
    \draw[thick,directed,bend right=30] (bottom2) to (b5) {};
    \draw[thick,directed,bend left=25] (bottom1) to (b1) {};
    \draw[thick,directed] (bottom1) to (b2) {};
    \draw[thick,directed,bend right=25] (bottom1) to (b3) {};
    \draw[thick,directed,bend left=25] (c4) to (top1) {};
    \draw[thick,directed] (c5) to (top1) {};
    \draw[thick,directed,bend right=25] (c6) to (top1) {};
	\draw[thick,directed,bend left=20] (c2) to (top2) {};
	\draw[thick,directed,bend left=5] (c3) to (top2) {};
	\draw[thick,directed,bend right=5] (top1) to (top2) {};
    \draw[thick,directed] (top2) to (d4) {};
    \draw[thick,directed] (c0) to (d0) {};
    \draw[thick,directed] (c1) to (d1) {};  
    \node[fill=none] at (0,3.25) {$G$};
\end{tikzpicture}
\hspace{.1in}
\raisebox{38pt}{\scalebox{2.5}{$\Rightarrow$}}
\hspace{.1in}
\begin{tikzpicture}
[scale=.47,auto=left,every node/.style={circle,fill=black,inner sep=1.0pt}]
    \node[draw] (a4) at (.3,0) {};
    \node[draw] (bottom2) at (.3,1) {};
    \node[draw] (bottom1) at (-.6,1.5) {};
    \node[fill=none] (b1) at (-1.2,2.6) {};
    \node[fill=none] (b2) at (-.6,2.6) {};
    \node[fill=none] (b3) at (0,2.6) {};
    \node[fill=none] (b4) at (.6,2.6) {};
    \node[fill=none] (b5) at (1.2,2.6) {};
	\node[fill=none] (c0) at (-1.8,3.9) {}; 
    \node[fill=none] (c1) at (-1.2,3.9) {};
    \node[fill=none] (c2) at (-.6,3.9) {};
    \node[fill=none] (c3) at (0,3.9) {};
    \node[fill=none] (c4) at (.6,3.9) {};
    \node[fill=none] (c5) at (1.2,3.9) {};
    \node[fill=none] (c6) at (1.8,3.9) {};
    \node[draw] (top1) at (1.2,5.0) {};
    \node[draw] (top2) at (.3,5.5) {};
    \node[draw] (top3) at (-.6,6.0) {};
    \node[draw] (d4) at (-.6,7.1) {};
    \draw[thick] (-2.1,2.5) rectangle (2.1,4); 
    \draw[thick,directed] (a4) to (bottom2) {};
    \draw[thick,directed,bend left=15] (bottom2) to (bottom1) {};
    \draw[thick,directed,bend right=10] (bottom2) to (b4) {};
    \draw[thick,directed,bend right=30] (bottom2) to (b5) {};
    \draw[thick,directed,bend left=25] (bottom1) to (b1) {};
    \draw[thick,directed] (bottom1) to (b2) {};
    \draw[thick,directed,bend right=25] (bottom1) to (b3) {};
    \draw[thick,directed,bend left=25] (c4) to (top1) {};
    \draw[thick,directed] (c5) to (top1) {};
    \draw[thick,directed,bend right=25] (c6) to (top1) {};
	\draw[thick,directed,bend left=20] (c2) to (top2) {};
	\draw[thick,directed,bend left=5] (c3) to (top2) {};
	\draw[thick,directed,bend right=5] (top1) to (top2) {};
    \draw[thick,directed,bend right=5] (top2) to (top3) {};
    \draw[thick,directed,bend left=20] (c0) to (top3) {};
    \draw[thick,directed,bend left=5] (c1) to (top3) {};
    \draw[thick,directed] (top3) to (d4) {};
    \node[fill=none] at (0,3.25) {$G$};
\end{tikzpicture}
\caption{A non-closed prograph $G \in \pc_5^3(n,n-1)$ and its justification $j(G) \in \pc^3(n + \frac{5-1}{3-1})$.}
\label{fig: justified prographs}
\end{figure}

\begin{theorem}
\label{thm: prograph vs tableaux bijection, non-rectangular}
Fix $n,m \geq 1$, $k \geq 2$, and take any $x \geq 1$ such that $x \equiv 1 \mod (k-1)$.  Then $\vert \pc_x^k(n,m) \vert = \vert \svt(\lambda/\mu,\rho) \vert$, where $\lambda = (n + \frac{x-1}{k-1},n + \frac{x-1}{k-1},m)$, $\mu = (\frac{x-1}{k-1},0,0)$, and $\rho = (1,k-1,1)$.
\end{theorem}
\begin{proof}
Let $j: \pc_x^k(n,m) \rightarrow \pc^k(n+\frac{x-1}{k-1})$ be justification and let $\Phi: \pc^k(n+\frac{x-1}{k-1}) \rightarrow \svt((n+\frac{x-1}{k-1})^3,\rho)$ be the forward bijection from Theorem \ref{thm: prograph vs tableaux bijection, rectangular}.  Then define $\chi: \svt((n+\frac{x-1}{k-1})^3,\rho) \rightarrow \svt(\lambda/\mu, \rho)$ as the map that deletes the first $\frac{x-1}{k-1}$ cells in the top row of $T \in \svt((n+\frac{x-1}{k-1})^3,\rho)$, deletes the last $\frac{y-1}{k-1}$ cells in the bottom row of $T$, and then reindexes all remaining entries so that no positive integers are skipped.  We define $\psi: \pc_x^k(n,m) \rightarrow \svt(\lambda/\mu,\rho)$ by $\psi = \chi \circ \Phi \circ j$, and show that $\psi$ is a bijection.  See Figure \ref{fig: non-closed bijection example} for an example of this map $\psi$.

Well-definedness of $j$, $\Phi$, and $\chi$ ensure that the composition $\psi$ is also well-defined.  To show that $\psi$ is a bijection, we begin showing that the restriction $\widetilde{\chi} = \chi \vert_{\im(\Phi \circ j)}$ is a bijection onto $\svt(\lambda / \mu, \rho)$.  So take any $G \in \pc_x^k(n,m)$.  Using Proposition \ref{thm: number of output strands}, the number of output strands in $G$ is $y = (k-1)(n-m)+x$.  Thus $n + \frac{x-1}{k-1} = m + \frac{y-1}{k-1}$, and $j(G) \in \pc^k(n + \frac{x-1}{k-1})$ is obtained from $G$ by recursively adding $\frac{x-1}{k-1}$ left-aligned coproducts to the bottom of $G$ and $\frac{y-1}{k-1}$ right-aligned products to the top of $G$.  Applying our depth-left first search to $j(G)$ then results in the first $\frac{x-1}{k-1}$ non-zero labels being applied to the leftmost children of the ``new" coproduct nodes at the bottom of $j(G)$, while the final $\frac{y-1}{k-1}$ labels are applied to the outputs of the ``new" product nodes at the top of $j(G)$.  This guarantees that the first row of every $T \in \im(\phi \circ j)$ begins with $1,\hdots,\frac{x-1}{k-1}$ and that the bottom row of every such $T$ ends with $k(n+\frac{x-1}{k-1})+m+1, \hdots, (k+1)(n+\frac{x-1}{k-1})$.  It follows that the entries deleted by $\chi$ are identical across all tableaux in $\widetilde{\chi}$, implying that $\widetilde{\chi}$ is a bijection.

Bijectivity of $\widetilde{\chi}$ implies that $\chi \circ \Phi$ is also bijective with inverse $(\chi \circ \Phi)^{-1} \equiv \Phi^{-1} \circ \widetilde{\chi}^{-1}$.  Notice that $\widetilde{\chi}^{-1}: \svt(\lambda/\mu,\rho) \rightarrow \svt((n+ \frac{x+1}{k-1})^3,\rho)$ is the function that reindexes all entries of $T \in \svt(\lambda/\mu,\rho)$ by $a \mapsto a+ \frac{x-1}{k-1}$, appends $\lbrace 1,\hdots,\frac{x-1}{k-1} \rbrace$ to the front of the top row, and appends the $\frac{y-1}{k-1}$ entries $\lbrace k(n+\frac{x-1}{k-1})+m+1, \hdots, (k+1)(n+\frac{x-1}{k-1}) \rbrace$ to the end of the bottom row.  This means that $\im(\Phi^{-1} \circ \widetilde{\chi}^{-1})$ are the prographs $G \in \pc^k(n+\frac{x-1}{k-1})$ with $\frac{x-1}{k-1}$ consecutive left-aligned coproducts at the bottom and $\frac{y-1}{k-1}$ consecutive right-aligned products at the top.

All of this allows us to define an ``unjustification" map $h: \im(\Phi^{-1} \circ \widetilde{\chi}^{-1}) \rightarrow \svt(\lambda/\mu, \rho)$ where, for any prograph $G \in \im(\Phi^{-1} \circ \widetilde{\chi}^{-1})$, one simply deletes the $\frac{x-1}{k-1}$ initial product nodes (along with their inputs) and deletes the $\frac{y-1}{k-1}$ final coproduct nodes (along with their outputs).  This map $h$ clearly satisfies $j \circ h(G) = G$ for any $G \in \im(\Phi^{-1} \circ \widetilde{\chi}^{-1})$ and $h \circ j(G) = G$ for any $G \in \svt(\lambda/\mu,\rho)$.  We may then conclude that $\psi$ is a bijection with inverse $(\chi \circ \Phi \circ j)^{-1} \equiv h \circ \Phi^{-1} \circ \widetilde{\chi}^{-1}$.
\end{proof}

\begin{figure}[ht!]
\centering
\raisebox{19pt}{
\begin{tikzpicture}
[scale=0.48,auto=left,every node/.style={circle,fill=black,inner sep=1pt}]
    \node[draw] (a1) at (.25,-1.4) {};
    \node[draw] (a2) at (3.55,-1.4) {};
    \node[draw] (a3) at (4.3,-1.4) {};
    \node[draw] (b) at (.25,0) {};
    \node[draw] (c) at (2.5,1) {};
    \node[draw] (d) at (1.5,2.5) {};
    \node[draw] (e) at (2.5,3.5) {};
    \node[draw] (f1) at (-0.5,4.9) {};
    \node[draw] (f2) at (0.25,4.9) {};    
    \node[draw] (f3) at (.75,4.9) {};
    \node[draw] (f4) at (1.5,4.9) {};
    \node[draw] (f5) at (2.5,4.9) {};
    \node[draw] (f6) at (3.55,4.9) {};
    \node[draw] (f7) at (4.3,4.9) {};
    \draw[thick,directed] (a1) to (b);
    \draw[thick,directed,bend left=10] (b) to (f1);
    \draw[thick,directed,bend left=0] (b) to (f2);
    \draw[thick,directed] (b) to (c);
    \draw[thick,directed] (c) to (d);
    \draw[thick,directed] (c) to (e);
    \draw[thick,directed,bend right=40] (c) to (e);
    \draw[thick,directed] (d) to (f4);
    \draw[thick,directed,bend left=10] (d) to (f3);
    \draw[thick,directed] (d) to (e);
    \draw[thick,directed] (e) to (f5);
	\draw[thick,directed] (a2) to (f6);
	\draw[thick,directed] (a3) to (f7);
\end{tikzpicture}}
\hspace{.1in} \raisebox{60pt}{\scalebox{3}{$\Rightarrow$} \kern-25pt{\raisebox{18pt}{$j$}}} \hspace{.2in}
\begin{tikzpicture}
[scale=0.48,auto=left,every node/.style={circle,fill=black,inner sep=1pt}]
	\node[draw] (base) at (3,-2.2) {};
	\node[draw] (coprod1) at (3,-1.1) {};
	\node[draw] (prod1) at (3.75,4.55) {};
	\node[draw] (prod2) at (2.25,5.45) {};
	\node[draw] (prod3) at (.5,6.35) {};
	\node[draw] (top) at (.5,7.45) {};
    \node[draw] (b) at (.25,0) {};
    \node[draw] (c) at (2.5,1) {};
    \node[draw] (d) at (1.4,2.5) {};
    \node[draw] (e) at (2.5,3.5) {};
    \node[fill=none] (f1) at (-1,4.9) {};
    \node[fill=none] (f2) at (0.25,4.9) {};    
    \draw[thick,directed] (base) to (coprod1);
    \draw[thick,directed,bend left=20] (coprod1) to (b);
    \draw[thick,directed,bend left=30] (b) to (prod3);
    \draw[thick,directed] (b) to (prod3);
    \draw[thick,directed] (b) to (c);
    \draw[thick,directed] (c) to (d);
    \draw[thick,directed] (c) to (e);
    \draw[thick,directed,bend right=50] (c) to (e);
    \draw[thick,directed] (d) to (prod2);
    \draw[thick,directed,bend left=50] (d) to (prod2);
    \draw[thick,directed] (d) to (e);
    \draw[thick,directed,bend left=15] (e) to (prod1);
	\draw[thick,directed,bend right=10] (coprod1) to (prod1);
	\draw[thick,directed,bend right=40] (coprod1) to (prod1);
	\draw[thick,directed,bend right=20] (prod1) to (prod2);
	\draw[thick,directed,bend right=20] (prod2) to (prod3);
	\draw[thick,directed] (prod3) to (top);
    \node[fill=none,scale=.85] (0) at (2.65,-1.65) {0};
    \node[fill=none,scale=.85] (1) at (1.7,-.6) {1};
    \node[fill=none,scale=.85] (2) at (-.8,3.9) {2};
    \node[fill=none,scale=.85] (3) at (.1,3.3) {3};
    \node[fill=none,scale=.85] (4) at (1.25,0.85) {4};
    \node[fill=none,scale=.85] (5) at (1.6,1.7) {5};
    \node[fill=none,scale=.85] (6) at (1.2,4.95) {6};
    \node[fill=none,scale=.85] (7) at (2.15,4.0) {7};
    \node[fill=none,scale=.85] (8) at (2.0,2.6) {8};
    \node[fill=none,scale=.85] (9) at (2.72,2.15) {9};
    \node[fill=none,scale=.85] (10) at (3.2,3.2) {10};
    \node[fill=none,scale=.85] (11) at (2.9,4.5) {11};
    \node[fill=none,scale=.85] (12) at (4.05,2) {12};
    \node[fill=none,scale=.85] (13) at (4.9,2.2) {13};
    \node[fill=none,scale=.85] (14) at (3.45,5.4) {14};
    \node[fill=none,scale=.85] (15) at (1.75,6.35) {15};
    \node[fill=none,scale=.85] (16) at (.0,6.95) {16};
\end{tikzpicture}
\hspace{.0in} \raisebox{60pt}{\scalebox{3}{$\Rightarrow$} \kern-27pt{\raisebox{17pt}{$\Phi$}}} \hspace{.15in}
\raisebox{66pt}{
\setlength{\tabcolsep}{2.5pt}
\scalebox{.85}{
\begin{tabular}{|>{$}c<{$}|>{$}c<{$}|>{$}c<{$}|>{$}c<{$}|}
\hline
\cellcolor{LightGray} 1 & 2 & 5 & 6 \\ \hline
3 \kern+3pt 4 & 7 \kern+3pt 8 & 9 \kern+3pt 10 & 12 \kern+3pt 13 \\ \hline
11 & \cellcolor{LightGray} 14 & \cellcolor{LightGray} 15 & \cellcolor{LightGray} 16 \\ \hline
\end{tabular}}}
\hspace{.1in} \raisebox{60pt}{\scalebox{3}{$\Rightarrow$} \kern-26pt{\raisebox{18pt}{$\chi$}}} \hspace{.15in}
\raisebox{66pt}{
\setlength{\tabcolsep}{2.5pt}
\scalebox{.85}{
\begin{tabular}{|>{$}c<{$}|>{$}c<{$}|>{$}c<{$}|>{$}c<{$}|}
\cline{2-4}
\multicolumn{1}{c|}{ } & 1 & 4 & 5 \\ \cline{1-4}
2 \kern+3pt 3 & 6 \kern+3pt 7 & 8 \kern+3pt 9 & 11 \kern+3pt 12 \\ \cline{1-4}
10 & \multicolumn{1}{c}{ } & \multicolumn{1}{c}{ } & \multicolumn{1}{c}{ } \\ \cline{1-1}
\end{tabular}}}
\caption{An example of the bijection $\psi: \pc_x^k(n,m) \rightarrow \svt(\lambda/\mu,\rho)$ for $k=3$, $x=3$, $n=3$, and $m=1$.}
\label{fig: non-closed bijection example}
\end{figure}
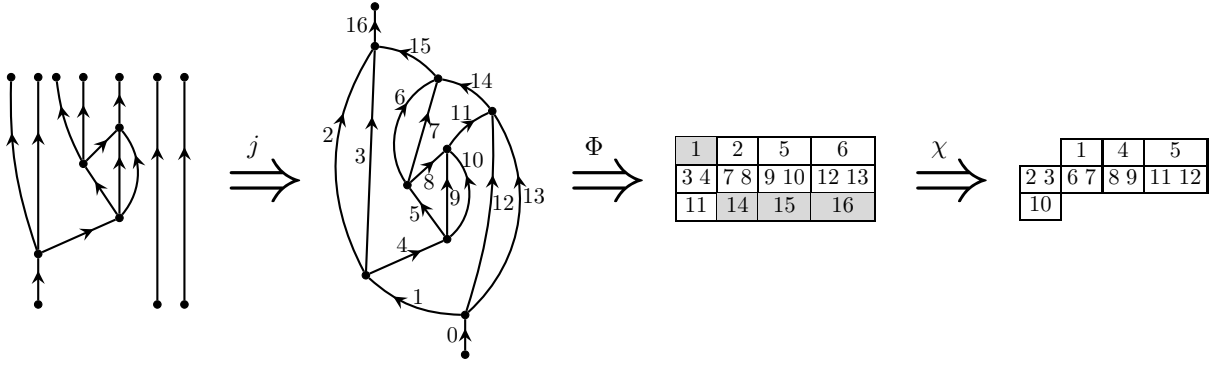

In light of Theorem \ref{thm: prograph vs tableaux bijection, non-rectangular}, one may define a modification of our depth-left first search that allows one to pass directly from an edge-labelling of $G \in \pc_x^k(n,m)$ to $\psi(G) \in \svt(\lambda/\mu,\rho)$, bypassing the justification and reindexing steps.  This $x$-fold depth-left first search is defined as below.

\begin{enumerate}
\item For any $G \in \pc_x^k(n,m)$ with $x \equiv 1 \mod(k-1)$, label the leftmost initial input of $G$ with the integer $0$.
\item After labelling the $i^{th}$ edge, determine the node subset $V_i$ from the depth-left first search of Section \ref{sec: prographs}.  If $V_i$ is non-empty, follow the procedure of Section \ref{sec: prographs} to find the edge labelled $i+1$.  If $V_i$ is empty, label the leftmost unlabelled initial input of $G$ with $i+1$
\end{enumerate}

Using the same terminology as Theorem \ref{thm: prograph vs tableaux bijection, non-rectangular}, let $\tau(G) \in \svt(\lambda/\mu,\rho)$ be the tableau that results from applying the $x$-fold depth-left first search to $G \in \pc_x^k(n,m)$, placing all integers labelling leftmost coproduct children of $G$ in the top row, placing all integers labelling product children of $G$ in the bottom row, and placing all remaining non-zero integers (including those labelling non-leftmost initial inputs of $G$) in the middle row.  This is in fact that same tableau that results from the composite bijection of Theorem \ref{thm: prograph vs tableaux bijection, non-rectangular}:

\begin{corollary}
\label{thm: prograph vs tableaux bijection, non-rectangular corollary}
Let $\psi: \pc_x^k(n,m) \rightarrow \svt(\lambda/\mu,\rho)$ be as in the proof of Theorem \ref{thm: prograph vs tableaux bijection, non-rectangular}.  For any $G \in \pc_x^k(n,m)$ with $x \equiv 1 \mod(k-1)$, $\tau(G) = \psi(G)$.
\end{corollary}
\begin{proof}
Recall that justification of $G$ introduces precisely $\frac{x-1}{k-1}$ leftmost coproduct children that receive the first $\frac{x-1}{k-1}$ nonzero labels in the depth-left first search on $j(G)$, as well as $\frac{y-1}{k-1}$ product children that receive the final $\frac{y-1}{k-1}$ labels the depth-left first search on $j(G)$.  As these are precisely the entries of $\Phi \circ j(G)$ that are deleted in the final stage of $\psi$, we merely need to argue that the depth-left first search of Section \ref{sec: prographs} labels the remaining edges of $j(G)$ in the same order that the $x$-fold depth-left first search labels the edges of $G$.  In particular, we need to show that the $i^{th}$ edge from the $x$-fold depth-left first search on $G$ corresponds to the $(i+\frac{x-1}{k-1})^{th}$ edge from our original depth-left first search on $j(G)$.

Inducting on $i$, consider the two alogorithms after the labelling of the $i^{th}$ edge of $G$.  If the set $V_i$ is non-empty for $G$, the set $V_{(i+\frac{x-1}{k-1})}$ is non-empty for $j(G)$.  Since the inputs to the initial coproduct nodes that appear only in $j(G)$ have lower edge labels than all other edges in $j(G)$, the element of $V_i$ in $G$ with the largest input corresponds to the element of $V_{(i+\frac{x-1}{k-1})}$ in $j(G)$ with the largest input.  This leads to equivalent placements of the next edge label in both graphs.  Now if $V_i$ is empty for $G$, it must be the case that $V_{(i+\frac{x-1}{k-1})}$ for $j(G)$ consists solely of nodes from the justification's $\frac{x-1}{k-1}$ initial coproducts.  As the edge labels on the inputs to these initial coproducts always decrease from left to right, the next edge labelled in $j(G)$ is always the leftmost output of the initial coproducts that has yet to be labelled.  These initial coproduct children of $j(G)$ precisely correspond to initial inputs in the non-justified graph $G$, implying that the next edge of $G$ to be labelled by the $x$-fold depth-left first search is the equivalent (non-leftmost) initial input of $G$.
\end{proof}

\subsection{The Sch\"utzenberger Involution}
\label{subsec: schutzenberger}

For any rectangular shape $\lambda \vdash N$, the Sch\"utzenberger involution is a map $f: S(\lambda) \rightarrow S(\lambda)$ that rotates $T \in S(\lambda)$ by 180 degrees and then renumbers entries via $a \mapsto N - a + 1$.  As described by Drube \cite{Drube}, one may define an analogue of the Sch\"utzenberger involution for standard set-valued Young tableaux.  For any rectangular shape $\lambda$ and row-constant density $\rho$, the set-valued Sch\"utzenberger involution $f : \svt(\lambda,\rho) \rightarrow  \svt(\lambda,\rho')$ is similarly defined via 180-degree rotation of $T \in \svt(\lambda,\rho)$, followed by a reversal in the order of entries in the resulting tableaux.  Here $\rho' = (\rho_m,\hdots,\rho_1)$ if $\rho = (\rho_1,\hdots,\rho_m)$, meaning only ``vertically symmetric" densities are preserved by $f$.

Now define a rotation operator $r: \pc_x^k(n,m) \rightarrow \pc_y^k(m,n)$ on (not-necessarily closed) $k$-ary prographs that corresponds to 180-degree rotation and a reversal in the orientation of all edges.  In Theorem \ref{thm: schutzenberger} we will show that a specialization of this operator to any closed $k$-ary prograph $G$ is compatible with the Sch\"utzenbeger involution on the associated set-valued tableaux $\Phi(G)$ from Theorem \ref{thm: prograph vs tableaux bijection, rectangular}, but first we need to analyze how rotation effects our edge-labelling algorithms.  It is in fact that case that the $x$-fold depth-left first search of Subsection \ref{subsec: non-closed prographs} labels the edges of $r(G) \in \pc_y^k(m,n)$ in an order that exactly reverses the order in which it labels the corresponding edges of $G \in \pc_x^k(n,m)$:

\begin{proposition}
\label{thm: rotated depth-left first search}
For any $k \geq 2$, $n,m \geq 0$, $x \geq 1$, set $N = x + kn + m - 1$ and consider the rotation operator $r: \pc_x^k(n,m) \rightarrow \pc_y^k(m,n)$.  For any edge $e$ of $G$, if the $x$-fold depth-left first search labels $e$ with the integer $i$, then the $x$-fold depth-left first search labels the corresponding edge of $r(G)$ with $N-i$.
\end{proposition}
\begin{proof}
We proceed by induction on the maximum edge label $N \geq 0$.  The $N = 0$ case is immediate, as both $G$ and $r(G)$ consist of a single edge labelled $0$.  For $N>0$, consider the edge $e$ of $G$ that receives the label $N$, which is always the rightmost output of $G$.  There are three options: 1) $e$ is a ``free strand" that does not originate at a product or coproduct, 2) $e$ is a product child, or 3) $e$ is a rightmost coproduct child.

If $e$ is a free strand, simply deleting $e$ produces a valid prograph $\widetilde{G}$ with maximal edge label $N-1$.  By the inductive hypothesis, the $x$-fold depth-left first search labels corresponding edges in $\widetilde{G}$ and $r(\widetilde{G})$ according to $i \mapsto N-1-i$.  Inserting a free strand (labelled $0$) on the left side of $G$ recovers $r(G)$, and effects our edge labelling in that the label of all edges in $r(\widetilde{G})$ are increased by $1$.  It follows that the $x$-fold depth-left first search labels corresponding edges in $G$ and $r(G)$ according to $i \mapsto N-i$.

If $e$ is a product child, we eliminate the product node at the source of $e$ as in the first row of Figure \ref{fig: edge labelling in rotation}, yielding a prograph $\widetilde{G}$ with $k-1$ additional outputs but maximal edge label $N-1$.  Applying the inductive hypothesis allows us to relate corresponding edge labels of $\widetilde{G}$ and $r(\widetilde{G})$ by $i \mapsto N-1-i$.  We then pre-compose $r(\widetilde{G})$ with an additional coproduct whose outputs are the $k$ leftmost inputs of $r(\widetilde{G})$, as in the top row of Figure \ref{fig: edge labelling in rotation}.  This recovers $r(G)$ and effects our edge labelling in that all edges apart from the new coproduct input are increased by $1$.  It follows that the $x$-fold search labels corresponding edges in $G$ and $r(G)$ according to $i \mapsto N-i$.

Lastly, if $e$ is a righmost coproduct child we eliminate the coproduct node at the source of $e$ as in the bottom row of Figure \ref{fig: edge labelling in rotation}, identifying the input of that coproduct with its leftmost output while extending all remaining outputs of to the bottom of the prograph as $k-1$ new inputs.  As the resulting graph $\widetilde{G}$ has maximal edge label $N-1$, we may once again relate corresponding edge labels of $\widetilde{G}$ and $r(\widetilde{G})$ by $i \mapsto N-1-i$.  Introducing a new product node into $r(\widetilde{G})$ as in the bottom row of Figure \ref{fig: edge labelling in rotation} recovers $r(G)$ and effects our $x$-fold search in such a way that the corresponding edges of $G$ and $r(G)$ are labelled according to $i \mapsto N-i$.
\end{proof}

\begin{figure}[ht!]
\centering
\begin{tikzpicture}
[scale=.55,auto=left,every node/.style={circle,fill=black,inner sep=1.0pt}]
    \node[draw] (top) at (0,4) {};
    \node[draw] (b) at (0,2.8) {};
    \node[fill=none] (a1) at (-1.3,1.5) {};
    \node[fill=none] (a2) at (1.3,1.5) {};
	\draw[thick] (-1.5,0) rectangle (1.5,1.6);
    \draw[thick,directed] (a1) to (b) {};
    \draw[thick,directed] (a2) to (b) {};
    \draw[thick,directed] (b) to (top) {};
    \node[fill=none] at (0,1.9) {$\cdots$};
    \node[fill=none] at (0,.8) {$\widetilde{G}$};
    \node[fill=none] at (.5,3.4) {\scalebox{.9}{$N$}};
    \node[fill=none] at (-1.2,2.1) {\scalebox{.9}{$a_1$}};
    \node[fill=none] at (1.55,2.15) {\scalebox{.9}{$N \kern-3pt- \kern-3pt1$}};
\end{tikzpicture}
\hspace{.05in} \raisebox{26pt}{\scalebox{3}{$\Rightarrow$}} \hspace{.05in}
\begin{tikzpicture}
[scale=.55,auto=left,every node/.style={circle,fill=black,inner sep=1.0pt}]
    \node[draw] (top1) at (-1.3,4) {};
    \node[draw] (top2) at (1.3,4) {};
    \node[fill=none] (a1) at (-1.3,1.5) {};
    \node[fill=none] (a2) at (1.3,1.5) {};
	\draw[thick] (-1.5,0) rectangle (1.5,1.6);
    \draw[thick,directed] (a1) to (top1) {};
    \draw[thick,directed] (a2) to (top2) {};
    \node[fill=none] at (0,2.7) {$\cdots$};
    \node[fill=none] at (0,.8) {$\widetilde{G}$};
    \node[fill=none] at (-1.7,2.5) {\scalebox{.9}{$a_1$}};
    \node[fill=none] at (2.05,2.5) {\scalebox{.9}{$N \kern-3pt- \kern-3pt1$}};
\end{tikzpicture}
\hspace{.05in} \raisebox{26pt}{\scalebox{3}{$\Rightarrow$} \kern-24pt{\raisebox{16pt}{$r$}}} \hspace{.25in}
\raisebox{-8pt}{\begin{tikzpicture}
[scale=.55,auto=left,every node/.style={circle,fill=black,inner sep=1.0pt}]
    \node[draw] (bottom1) at (-1.3,0) {};
    \node[draw] (bottom2) at (1.3,0) {};
    \node[fill=none] (a1) at (-1.3,2.5) {};
    \node[fill=none] (a2) at (1.3,2.5) {};
	\draw[thick] (-1.5,2.4) rectangle (1.5,4);
    \draw[thick,directed] (bottom1) to (a1) {};
    \draw[thick,directed] (bottom2) to (a2) {};
    \node[fill=none] at (0,1.3) {$\cdots$};
    \node[fill=none] at (0,3.2) {$r(\widetilde{G})$};
    \node[fill=none] at (-1.65,1.0) {\scalebox{.9}{$0$}};
    \node[fill=none] at (2.85,1.0) {\scalebox{.9}{$(N \kern-3pt- \kern-3pt1) \kern-1pt- \kern-1pta_1$}};
\end{tikzpicture}}
\hspace{-.2in} \raisebox{26pt}{\scalebox{3}{$\Rightarrow$}} \hspace{.15in}
\raisebox{-2pt}{\begin{tikzpicture}
[scale=.55,auto=left,every node/.style={circle,fill=black,inner sep=1.0pt}]
    \node[draw] (bottom) at (0,0) {};
    \node[draw] (b) at (0,1.2) {};
    \node[fill=none] (a1) at (-1.3,2.5) {};
    \node[fill=none] (a2) at (1.3,2.5) {};
	\draw[thick] (-1.5,2.4) rectangle (1.5,4);
    \draw[thick,directed] (b) to (a1) {};
    \draw[thick,directed] (b) to (a2) {};
    \draw[thick,directed] (bottom) to (b) {};
    \node[fill=none] at (0,2.1) {$\cdots$};
    \node[fill=none] at (0,3.2) {$r(\widetilde{G})$};
    \node[fill=none] at (.4,.5) {\scalebox{.9}{$0$}};
    \node[fill=none] at (-.95,1.5) {\scalebox{.9}{$1$}};
    \node[fill=none] at (1.5,1.5) {\scalebox{.9}{$N \kern-3pt- \kern-3pta_1$}};
\end{tikzpicture}}

\vspace{.15in}

\begin{tikzpicture}
[scale=.55,auto=left,every node/.style={circle,fill=black,inner sep=1.0pt}]
	\node[fill=none] (a) at (0,1.4) {}; 
    \node[draw] (b) at (0,2.2) {};
    \node[fill=none] (c1) at (-2.4,3.7) {};
    \node[fill=none] (c2) at (0,3.7) {};
    \node[draw] (top) at (2,5) {};
	\draw[thick] (-1.5,.1) rectangle (1.5,1.5);
	\draw[thick] (-2.7,3.6) rectangle (.3,5.0);
    \draw[thick,directed] (a) to (b) {};
    \draw[thick,directed,bend left=25] (b) to (c1) {};
    \draw[thick,directed] (b) to (c2) {};
	\draw[thick,directed,bend right=40] (b) to (top) {};
	\node[fill=none] at (-.9,3.1) {$\cdots$};
	\node[fill=none] at (0,.8) {$G_1$};
    \node[fill=none] at (-1.2,4.2) {$G_2$};
    \node[fill=none] at (0.75,1.8) {\scalebox{.9}{$b_1 \kern-3pt- \kern-3pt1$}};
    \node[fill=none] at (-1.9,2.5) {\scalebox{.9}{$b_1$}};
    \node[fill=none] at (0.65,3.1) {\scalebox{.9}{$b_{k\kern-1pt-\kern-1pt1}$}};
    \node[fill=none] at (1.8,2.9) {\scalebox{.9}{$N$}};
\end{tikzpicture}
\hspace{.0in} \raisebox{32pt}{\scalebox{3}{$\Rightarrow$}} \hspace{.0in}
\begin{tikzpicture}
[scale=.55,auto=left,every node/.style={circle,fill=black,inner sep=1.0pt}]
	\node[fill=none] (a) at (-2.4,1.4) {}; 
    \node[fill=none] (c1) at (-2.4,3.7) {};
    \node[fill=none] (cnew) at (-2.0,3.7) {};
    \node[fill=none] (c2) at (-.1,3.7) {};
    \node[draw] (top) at (2,5) {};
    \node[draw] (bottom1) at (-.4,0) {};
    \node[draw] (bottom2) at (1.5,0) {};
    \node[draw] (bottom3) at (2,0) {};
	\draw[thick] (-3.9,.1) rectangle (-.9,1.5);
	\draw[thick] (-2.7,3.6) rectangle (.3,5.0);
    \draw[thick,directed] (a) to (c1) {};
    \draw[thick,directed,bend right=15] (bottom1) to (cnew) {};
    \draw[thick,directed,bend right=15] (bottom2) to (c2) {};
    \draw[thick,directed] (bottom3) to (top) {};
	\node[fill=none] at (-.2,2.5) {$\cdots$};
	\node[fill=none] at (-2.4,.8) {$G_1$};
    \node[fill=none] at (-1.2,4.2) {$G_2$};
    \node[fill=none] at (-3.0,2.3) {\scalebox{.9}{$b_1 \kern-3pt- \kern-3pt1$}};
    \node[fill=none] at (.2,.5) {\scalebox{.9}{$b_2 \kern-3pt- \kern-3pt1$}};
    \node[fill=none] at (.45,1.3) {\scalebox{.85}{$b_{k\kern-1pt-\kern-1pt1} \kern-4pt- \kern-3pt1$}};
    \node[fill=none] at (1.35,3.8) {\scalebox{.9}{$N\kern-3pt-\kern-3pt1$}};
\end{tikzpicture}
\hspace{.05in} \raisebox{32pt}{\scalebox{3}{$\Rightarrow$} \kern-24pt{\raisebox{16pt}{$r$}}} \hspace{.25in}
\begin{tikzpicture}
[scale=.55,auto=left,every node/.style={circle,fill=black,inner sep=1.0pt}]
	\node[fill=none] (a) at (2.4,-1.4) {}; 
    \node[fill=none] (c1) at (2.4,-3.7) {};
    \node[fill=none] (cnew) at (2.0,-3.7) {};
    \node[fill=none] (c2) at (.1,-3.7) {};
    \node[draw] (top) at (-2,-5) {};
    \node[draw] (bottom1) at (.4,0) {};
    \node[draw] (bottom2) at (-1.5,0) {};
    \node[draw] (bottom3) at (-2,0) {};
	\draw[thick] (3.9,-.1) rectangle (.9,-1.5);
	\draw[thick] (2.7,-3.6) rectangle (-.3,-5.0);
    \draw[thick,directed] (c1) to (a) {};
    \draw[thick,directed,bend left=15] (cnew) to (bottom1) {};
    \draw[thick,directed,bend left=15] (c2) to (bottom2) {};
    \draw[thick,directed] (top) to (bottom3) {};
	\node[fill=none] at (.2,-2.5) {$\cdots$};
	\node[fill=none] at (2.4,-.8) {$r(G_1)$};
    \node[fill=none] at (1.2,-4.35) {$r(G_2)$};
    \node[fill=none] at (3.2,-2.8) {\scalebox{.9}{$N \kern-3pt- \kern-3ptb_1$}};
    \node[fill=none] at (.15,-1.9) {\scalebox{.9}{$N \kern-3pt- \kern-3ptb_2$}};
    \node[fill=none] at (-.49,-0.5) {\scalebox{.83}{$N \kern-4pt- \kern-3ptb_{k \kern-1pt- \kern-1pt1}$}};
    \node[fill=none] at (-1.75,-3.0) {\scalebox{.9}{$0$}};
\end{tikzpicture}
\hspace{.1in} \raisebox{32pt}{\scalebox{3}{$\Rightarrow$}} \hspace{.0in}
\begin{tikzpicture}
[scale=.55,auto=left,every node/.style={circle,fill=black,inner sep=1.0pt}]
	\node[fill=none] (a) at (0,-1.4) {}; 
    \node[draw] (b) at (0,-2.2) {};
    \node[fill=none] (c1) at (2.4,-3.7) {};
    \node[fill=none] (c2) at (0,-3.7) {};
    \node[draw] (top) at (-2,-5) {};
	\draw[thick] (1.5,-.1) rectangle (-1.5,-1.5);
	\draw[thick] (2.7,-3.6) rectangle (-.3,-5.0);
    \draw[thick,directed] (b) to (a) {};
    \draw[thick,directed,bend right=25] (c1) to (b) {};
    \draw[thick,directed] (c2) to (b) {};
	\draw[thick,directed,bend left=40] (top) to (b) {};
	\node[fill=none] at (.8,-2.8) {$\cdots$};
	\node[fill=none] at (0,-.8) {$r(G_1)$};
    \node[fill=none] at (1.2,-4.3) {$r(G_2)$};
    \node[fill=none] at (-1.15,-1.9) {\scalebox{.9}{$N \kern-3pt- \kern-3ptb_1 \kern-3pt+ \kern-3pt1$}};
    \node[fill=none] at (2.0,-2.4) {\scalebox{.9}{$N \kern-3pt- \kern-3ptb_1$}};
    \node[fill=none] at (0.9,-3.35) {\scalebox{.83}{$N \kern-3pt- \kern-3ptb_{k\kern-1pt-\kern-1pt1}$}};
    \node[fill=none] at (-1.8,-3.1) {\scalebox{.9}{$0$}};
\end{tikzpicture}
\caption{The effect of the rotation operator upon edge labels surrounding the final product node or coproduct node of $G \in \pc_x^k(n,m)$, utilizing the ``resolution" techniques from the proof of Proposition \ref{thm: rotated depth-left first search}.  In the top row, $\widetilde{G}$ may include additional output edges that lie to the left of the edge labelled $N$.}
\label{fig: edge labelling in rotation}
\end{figure}
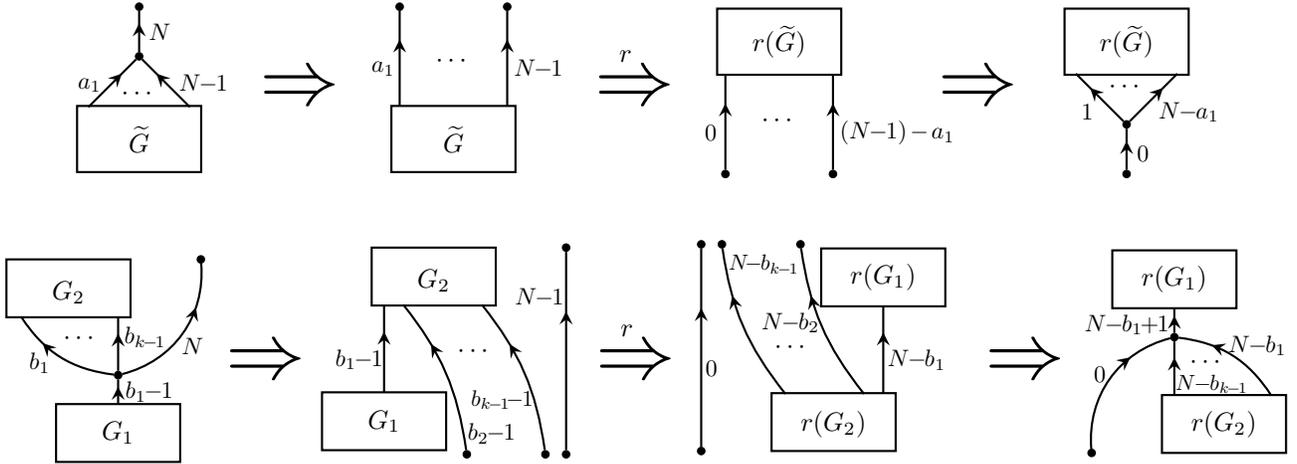

In the case of $x=1$ and $n=m$, the rotation operator reduces to an involution $r:\pc^k(n) \rightarrow \pc^k(n)$ of closed $k$-ary prographs.  Proposition \ref{thm: rotated depth-left first search} then states that the depth-left first search of Section \ref{sec: prographs} relates corresponding edges of $G$ and $r(G)$ according to $i \mapsto (k+1)n - i$.  This allows us to derive the following relationship between the rotation operator on closed $k$-ary prographs, the Sch\"utzenberger involution on rectangular set-valued tableaux, and the bijection $\Phi$ from Theorem \ref{thm: prograph vs tableaux bijection, rectangular}.  See Figure \ref{fig: schutzenberger versus rotation} for an example of this compatibility.

\begin{theorem}
\label{thm: schutzenberger}
Fix $k \geq 2$ and $n \geq 1$, and let $\rho = (1,k-1,1)$.  For $\Phi : \pc^k(n) \rightarrow \svt(\lambda,\rho)$ defined as in Theorem \ref{thm: prograph vs tableaux bijection, rectangular}, the rotation operator $r: \pc^k(n) \rightarrow \pc^k(n)$, and the set-valued Sch\"utzenberger involution $f : \svt(\lambda,\rho) \rightarrow  \svt(\lambda,\rho)$, we have $\Phi \circ r = f \circ \Phi$.
\end{theorem}
\begin{proof}
Take arbitrary $G \in \pc^k(n)$ and set $N = (k+1)n$, so that $G$ contains $N+1$ total edges and the cells of $\Phi(G) \in \svt(\lambda,\rho)$ are filled with $\lbrace 1,\hdots,N \rbrace$.  We show that leftmost coproduct children in $G$ correspond to bottom row entries in both $\Phi \circ r (G)$ and $f \circ \Phi(G)$, while product outputs in $G$ correspond to top row entries in both $\Phi \circ r (G)$ and $f \circ \Phi(G)$.  This implies that $\Phi \circ r (G)$ and $f \circ \Phi(G)$ feature identical sequences of integers across their top and bottom rows, and hence must be the same tableau.

So assume $G$ has been labelled according to our depth-left first search.  By Proposition \ref{thm: rotated depth-left first search}, if an edge in $G$ is labelled with the integer $a$, then the corresponding edge in $r(G)$ is labelled with $N-a$.  The depth-left first search is defined in such a way that $a$ labels a leftmost coproduct output in $G$ if and only if $a-1$ labels the input to the same coproduct node for which $a$ labels the leftmost child.  As demonstrated in the left side of Figure \ref{fig: schutzenberger proof cases}, this means that $N-(a-1)$ labels a product output in the rotated prograph $r(G)$.  It follows that $N-a+1$ appears in the bottom row of $\Phi \circ r(G)$.  On the other hand, $a$ being a leftmost coproduct child implies that $a$ appears in the top row of $\Phi(G)$, and hence that $N-a+1$ appears in the bottom row of $f \circ \Phi(G)$.

As $r$ is an involution, the case where $a$ labels a product in $G$ follows directly from reversing the roles of $G$ and $r(G)$ in the previous paragraph.  See the right side of Figure \ref{fig: schutzenberger proof cases} for a demonstration.  In this case we may conclude that $N-a+1$ appears in the top row of both $\Phi \circ r(G)$ and $f \circ \Phi(G)$, as required.
\end{proof}

\begin{figure}[ht!]
\centering
\begin{tikzpicture}
[scale=.7,auto=left,every node/.style={circle,fill=black,inner sep=1.15pt}]
	\node[fill=none] (b) at (0,0) {};
	\node[draw] (m) at (0,1) {};
	\node[fill=none] (t1) at (-1,2) {};
	\node[fill=none] (t2) at (1,2) {}; 
	\node[fill=none] (t3) at (0.4,2) {};
    \node[fill=none] (t4) at (-0.4,2) {};
	\draw[thick] (m) to (t3);
	\draw[thick] (b) to (m);
	\node[fill=none,scale=.9] at (-0.5,0.5) {$a-1$};
	\draw[thick] (m) to (t1);
	\node[fill=none,scale=.9] at (-0.6,1.3) {$a$};
	\draw[thick] (m) to (t4);
	\draw[thick] (m) to (t2);
\end{tikzpicture}
\hspace{0.1in}
\raisebox{21pt}{\scalebox{2}{$\curvearrowright$}}
\hspace{0.1in}
\begin{tikzpicture}
[scale=.7,auto=left,every node/.style={circle,fill=black,inner sep=1.15pt}]
	\node[fill=none] (b) at (0,2) {};
	\node[draw] (m) at (0,1) {};
	\node[fill=none] (t1) at (-1,0) {};
	\node[fill=none] (t2) at (1,0) {};
	\node[fill=none] (t3) at (-.4,0) {};
	\node[fill=none] (t4) at (.4,0) {};
	\draw[thick] (m) to (t3) {};
	\draw[thick] (b) to (m);
	\node[fill=none,scale=.9] at (.95,1.5) {$N-a+1$};
	\draw[thick] (m) to (t1);
	\draw[thick] (m) to (t2);
	\node[fill=none,scale=.9] at (1.2,.55) {$N-a$};
	\draw[thick] (m) to (t4);
\end{tikzpicture}
\hspace{.75in}
\begin{tikzpicture}
[scale=.7,auto=left,every node/.style={circle,fill=black,inner sep=1.15pt}]
	\node[fill=none] (b) at (0,2) {};
	\node[draw] (m) at (0,1) {};
	\node[fill=none] (t1) at (-1,0) {};
	\node[fill=none] (t2) at (1,0) {};
	\node[fill=none] (t3) at (-.4,0) {};
	\node[fill=none] (t4) at (.4,0) {};
	\draw[thick] (m) to (t3) {};
	\draw[thick] (b) to (m);
	\node[fill=none,scale=.9] at (.2,1.5) {$a$};
	\draw[thick] (m) to (t1);
	\draw[thick] (m) to (t2);
	\node[fill=none,scale=.9] at (1.1,.6) {$a-1$};
	\draw[thick] (m) to (t4);
\end{tikzpicture}
\hspace{0.1in}
\raisebox{21pt}{\scalebox{2}{$\curvearrowright$}}
\hspace{0.0in}
\begin{tikzpicture}
[scale=.7,auto=left,every node/.style={circle,fill=black,inner sep=1.15pt}]
	\node[fill=none] (b) at (0,0) {};
	\node[draw] (m) at (0,1) {};
	\node[fill=none] (t1) at (-1,2) {};
	\node[fill=none] (t2) at (1,2) {};
	\node[fill=none] (t3) at (0.4,2) {};
    \node[fill=none] (t4) at (-0.4,2) {};
	\draw[thick] (m) to (t3);
	\draw[thick] (b) to (m);
	\node[fill=none,scale=.9] at (-0.65,0.5) {$N-a$};
	\draw[thick] (m) to (t1);
	\node[fill=none,scale=.9] at (-1.5,1.4) {$N-a+1$};
	\draw[thick] (m) to (t4);
	\draw[thick] (m) to (t2);
\end{tikzpicture}
\caption{A demonstration of how the edge labels of leftmost coproduct outputs (left) and product outputs (right) behave under 180-degree rotation of the underlying prograph.}
\label{fig: schutzenberger proof cases}
\end{figure}
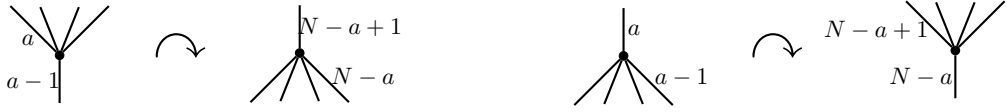

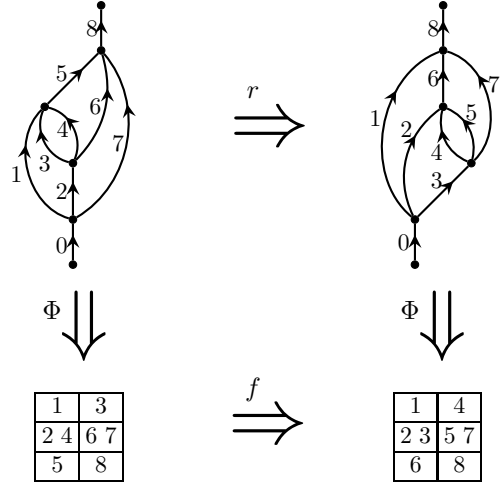
\begin{figure}[ht!]
\centering
\begin{subfigure}[b]{0.15\textwidth}
\centering
\begin{tikzpicture}
[scale=.75,auto=left,every node/.style={circle,fill=black,inner sep=1.0pt}]
    \node[draw] (a) at (0,.2) {};
    \node[draw] (b) at (0,1) {};
    \node[draw] (c) at (0,2) {};
    \node[draw] (d) at (-.5,3) {};
    \node[draw] (e) at (.5,4) {};
    \node[draw] (f) at (.5,4.8) {};
    \node[fill=none,scale=.9] (0) at (-.2,.55) {0};
    \node[fill=none,scale=.9] (1) at (-1,1.8) {1};
    \node[fill=none,scale=.9] (2) at (-.2,1.5) {2};
    \node[fill=none,scale=.9] (3) at (-.5,2.0) {3};
    \node[fill=none,scale=.9] (4) at (-.18,2.6) {4};
    \node[fill=none,scale=.9] (5) at (-.2,3.6) {5};
    \node[fill=none,scale=.9] (6) at (.4,3) {6};
    \node[fill=none,scale=.9] (7) at (.8,2.35) {7};
    \node[fill=none,scale=.9] (8) at (.34,4.4) {8};
    \draw[thick,directed] (a) to (b);
    \draw[thick,directed,bend left=60] (b) to (d);
    \draw[thick,directed,bend right=50] (b) to (e);
    \draw[thick,directed,bend left=50] (c) to (d);
    \draw[thick,directed,bend right=30] (c) to (e);
    \draw[thick,directed] (b) to (c);
    \draw[thick,directed,bend right=50] (c) to (d);
    \draw[thick,directed] (d) to (e);
    \draw[thick,directed] (e) to (f);
\end{tikzpicture}

\vspace{.1in}

\raisebox{10pt}{$\Phi$} \kern-4pt \scalebox{3}{$\Downarrow$}

\vspace{.15in}

\setlength{\tabcolsep}{2.5pt}
\scalebox{.9}{
\begin{tabular}{|>{$}c<{$}|>{$}c<{$}|}
\hline
1 & 3 \\ \hline
2 \kern+3pt 4 & 6 \kern+3pt 7 \\ \hline
5 & 8 \\ \hline
\end{tabular}}

\end{subfigure}
\begin{subfigure}[b]{0.1\textwidth}
\centering
\scalebox{3}{$\Rightarrow$} \kern-27pt \raisebox{18pt}{$r$}

\vspace{1.2in}

\scalebox{3}{$\Rightarrow$} \kern-27pt \raisebox{18pt}{$f$}
\end{subfigure}
\begin{subfigure}[b]{0.18\textwidth}
\centering
\begin{tikzpicture}
[scale=.75,auto=left,every node/.style={circle,fill=black,inner sep=1.0pt}]
    \node[draw] (a) at (0,.2) {};
    \node[draw] (b) at (0,1) {};
    \node[draw] (c) at (1,2) {};
    \node[draw] (d) at (.5,3) {};
    \node[draw] (e) at (.5,4) {};
    \node[draw] (f) at (.5,4.8) {};
    \node[fill=none,scale=.9] (0) at (-.2,.6) {0};
    \node[fill=none,scale=.9] (1) at (-.7,2.8) {1};
    \node[fill=none,scale=.9] (2) at (-.15,2.6) {2};
    \node[fill=none,scale=.9] (3) at (.39,1.7) {3};
    \node[fill=none,scale=.9] (4) at (.39,2.2) {4};
    \node[fill=none,scale=.9] (5) at (1,2.9) {5};
    \node[fill=none,scale=.9] (6) at (.33,3.5) {6};
    \node[fill=none,scale=.9] (7) at (1.4,3.4) {7};
    \node[fill=none,scale=.9] (8) at (.33,4.4) {8};
    \draw[thick,directed] (a) to (b);
    \draw[thick,directed] (b) to (c);
    \draw[thick,directed,bend left=40] (b) to (d);
    \draw[thick,directed,bend left=60] (b) to (e);
    \draw[thick,directed,bend left=40] (c) to (d);
    \draw[thick,directed,bend right=40] (c) to (d);
    \draw[thick,directed,bend right=60] (c) to (e);
    \draw[thick,directed] (d) to (e);
    \draw[thick,directed] (e) to (f);
\end{tikzpicture}

\vspace{.1in}

\raisebox{10pt}{$\Phi$} \kern-4pt \scalebox{3}{$\Downarrow$}

\vspace{.15in}

\setlength{\tabcolsep}{2.5pt}
\scalebox{.9}{
\begin{tabular}{|>{$}c<{$}|>{$}c<{$}|}
\hline
1 & 4 \\ \hline
2 \kern+3pt 3 & 5 \kern+3pt 7 \\ \hline
6 & 8 \\ \hline
\end{tabular}}
\end{subfigure}
\caption{An example of the relationship between rotation $r$ of $k$-ary product-coproduct prographs and the generalized Sch\"utzenberger involution $f$ on standard set-valued Young tableaux.}
\label{fig: schutzenberger versus rotation}
\end{figure}

As Proposition \ref{thm: rotated depth-left first search} applies to all $x$-fold $k$-ary prographs, the result of Theorem \ref{thm: schutzenberger} may be directly extended to non-closed prographs if one defines a suitable generalization of the Sch\"utzenberger involution.  If $\lambda = (n,n,n-a)$ and $\mu = (b,0,0)$, let $\lambda'=(n,n,n-b)$ and $\mu'=(a,0,0)$.  Then there exists a map $F:\svt(\lambda/\mu,\rho) \rightarrow \svt(\lambda',\mu',\rho)$ that is defined via $180$-degree rotation of $T \in \svt(\lambda/\mu,\rho)$ and a reversal $i \mapsto 3n-a-b+1-i$ of entries in the resulting tableau.  This map clearly specializes to the Sch\"utzenberger involution when $a=b=0$, and a superficial modification of the technique from Theorem \ref{thm: schutzenberger} yields $\psi \circ r = F \circ \psi$.  Notice how $F$ flips the number of ``missing boxes" in the top and bottom rows of a skew set-valued tableau, similarly to how $r$ flips the number of ``missing" products and coproducts needed to justify the associated $x$-fold prograph.

\section{Future Directions}
\label{sec: future directions}

\subsection{Non-Closed $k$-ary Prographs $\pc_x^k(n,m)$ for which $x \not\equiv 1 \mod(k-1)$}
\label{subsec: x not 1 mod(k-1)}

Subsection \ref{subsec: non-closed prographs} entirely restricted its attention to sets $\pc_x^k(n,m)$ of $x$-fold $k$-ary prographs for which $x \equiv 1 \mod(k-1)$.  Developing an analogue to Theorem \ref{thm: prograph vs tableaux bijection, non-rectangular} in the case of $x \not\equiv 1 \mod(k-1)$ is significantly more involved, as such prographs require a modification of the justification operator whose effect on the associated set-valued tableaux is more difficult to interpret.  Although we stop short of proving an explicit bijection, we pause to outline how the techniques of Subsection \ref{subsec: non-closed prographs} may be generalized to the case of general $\pc_x^k(n,m)$.

So let $x \equiv a \mod(k-1)$, where $2 \leq a \leq k-1$, and consider the set $\pc_x^k(n,m)$.  There exists an injection $J : \pc_x^k(n,m) \rightarrow \pc^k(n+ \frac{x+k-a-1}{k-1})$ in which $k-a$ free strands are added on the left side of $G \in \pc_x^k(n,m)$, producing a prograph $\widetilde{G} \in \pc_{x+k-a}^k(n,m)$ in which the number of inputs is $1 \kern-3pt\mod(k-1)$, and then the original justification operator $j$ is applied to $\widetilde{G}$.  For $G \in \pc_x^k(n,m)$, we call the image $J(G) \in \pc^k(n + \frac{x+k-a-1}{k-1})$ the \textbf{left-weighted justification} of $G$.  See Figure \ref{fig: justified prographs, non-rectangular} for an example of left-weighted justification.

\begin{figure}[ht!]
\centering
\begin{tikzpicture}
[scale=.47,auto=left,every node/.style={circle,fill=black,inner sep=1.0pt}]
    \node[draw] (a1) at (-.4,0) {};
    \node[draw] (a2) at (.4,0) {};
    \node[fill=none] (b1) at (-.4,2.6) {};
    \node[fill=none] (b2) at (.4,2.6) {};
	\node[fill=none] (c0) at (-1.4,3.9) {};
    \node[fill=none] (c1) at (-.7,3.9) {};
    \node[fill=none] (c2) at (0,3.9) {};
    \node[fill=none] (c3) at (.7,3.9) {};
    \node[fill=none] (c4) at (1.4,3.9) {};
   	\node[draw] (d0) at (-1.4,7.1) {};  
    \node[draw] (d1) at (-.7,7.1) {};
    \node[draw] (d2) at (0,7.1) {};
    \node[draw] (d3) at (0.7,7.1) {};
    \node[draw] (d4) at (1.4,7.1) {};
	\draw[thick] (-1.7,2.5) rectangle (1.7,4);
    \draw[thick,directed] (a1) to (b1) {};
    \draw[thick,directed] (a2) to (b2) {};
    \draw[thick,directed] (c0) to (d0) {};    
    \draw[thick,directed] (c1) to (d1) {};
    \draw[thick,directed] (c2) to (d2) {};
    \draw[thick,directed] (c3) to (d3) {};
    \draw[thick,directed] (c4) to (d4) {};
    \node[fill=none] at (0,3.25) {$G$};
\end{tikzpicture}
\hspace{.1in}
\raisebox{38pt}{\scalebox{3}{$\Rightarrow$}}
\hspace{.1in}
\begin{tikzpicture}
[scale=.47,auto=left,every node/.style={circle,fill=black,inner sep=1.0pt}]
    \node[draw] (a1) at (-.4,0) {};
    \node[draw] (a2) at (.4,0) {};
    \node[fill=none] (b1) at (-.4,2.6) {};
    \node[fill=none] (b2) at (.4,2.6) {};
	\node[fill=none] (c0) at (-1.4,3.9) {};
    \node[fill=none] (c1) at (-.7,3.9) {};
    \node[fill=none] (c2) at (0,3.9) {};
    \node[fill=none] (c3) at (.7,3.9) {};
    \node[fill=none] (c4) at (1.4,3.9) {};
   	\node[draw] (d0) at (-1.4,7.1) {};  
    \node[draw] (d1) at (-.7,7.1) {};
    \node[draw] (d2) at (0,7.1) {};
    \node[draw] (d3) at (0.7,7.1) {};
    \node[draw] (d4) at (1.4,7.1) {};
    \node[draw] (bottom1) at (-3.1,0) {};
    \node[draw] (bottom2) at (-2.4,0) {};
    \node[draw] (top1) at (-3.1,7.1) {};
    \node[draw] (top2) at (-2.4,7.1) {};
	\draw[thick] (-1.7,2.5) rectangle (1.7,4);
    \draw[thick,directed] (a1) to (b1) {};
    \draw[thick,directed] (a2) to (b2) {};
    \draw[thick,directed] (c0) to (d0) {};    
    \draw[thick,directed] (c1) to (d1) {};
    \draw[thick,directed] (c2) to (d2) {};
    \draw[thick,directed] (c3) to (d3) {};
    \draw[thick,directed] (c4) to (d4) {};
    \draw[thick,directed] (bottom1) to (top1) {};
    \draw[thick,directed] (bottom2) to (top2) {};
    \node[fill=none] at (0,3.25) {$G$};
\end{tikzpicture}
\hspace{.1in} \raisebox{38pt}{\scalebox{3}{$\Rightarrow$} \kern-25pt{\raisebox{18pt}{$j$}}} \hspace{.25in}
\begin{tikzpicture}
[scale=.47,auto=left,every node/.style={circle,fill=black,inner sep=1.0pt}]
    \node[draw] (bottom0) at (-1.9,0) {};
    \node[draw] (bottom1) at (-1.9,1) {};
    \node[draw] (top2) at (0.35,5.2) {};
    \node[draw] (top1) at (-1.9,6.1) {};
    \node[draw] (top0) at (-1.9,7.1) {};
    \node[fill=none] (b1) at (-.4,2.6) {};
    \node[fill=none] (b2) at (.4,2.6) {};
	\node[fill=none] (c0) at (-1.4,3.9) {};
    \node[fill=none] (c1) at (-.7,3.9) {};
    \node[fill=none] (c2) at (0,3.9) {};
    \node[fill=none] (c3) at (.7,3.9) {};
    \node[fill=none] (c4) at (1.4,3.9) {};
   	\node[fill=none] (d0) at (-1.4,7.1) {};  
    \node[fill=none] (d1) at (-.7,7.1) {};
    \node[fill=none] (d2) at (0,7.1) {};
    \node[fill=none] (d3) at (0.7,7.1) {};
    \node[fill=none] (d4) at (1.4,7.1) {};
	\draw[thick] (-1.7,2.5) rectangle (1.7,4);
    \draw[thick,directed] (bottom0) to (bottom1) {};
    \draw[thick,directed,bend right=15] (bottom1) to (b1) {};
    \draw[thick,directed,bend right=30] (bottom1) to (b2) {};
    \draw[thick,directed,bend left=20] (c1) to (top2) {};
    \draw[thick,directed,bend left=10] (c2) to (top2) {};
    \draw[thick,directed,bend right=10] (c3) to (top2) {};
    \draw[thick,directed,bend right=20] (c4) to (top2) {};
	\draw[thick,directed,bend right=10] (c0) to (top1) {};
	\draw[thick,directed,bend right=5] (top2) to (top1) {};
    \draw[thick,directed,bend left=10] (bottom1) to (top1) {};
    \draw[thick,directed,bend left=40] (bottom1) to (top1) {};
    \draw[thick,directed] (top1) to (top0) {};
    \node[fill=none] at (0,3.25) {$G$};
\end{tikzpicture}
\caption{A non-closed prograph $G \in \pc_2^4(n,n-1)$ and its left-weighted justification $J(G) \in \pc^4(n + \frac{2+4-2-1}{4-1})$.}
\label{fig: justified prographs, non-rectangular}
\end{figure}
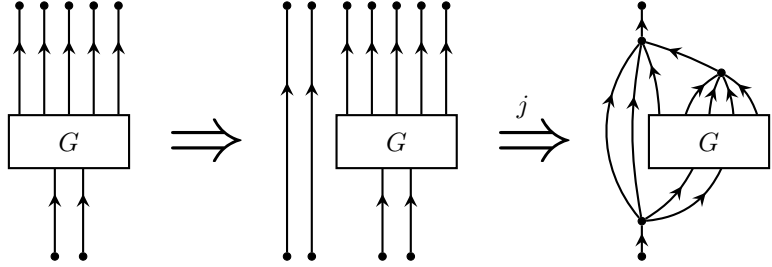

Following the techniques of Theorem \ref{thm: prograph vs tableaux bijection, non-rectangular}, left-weighted justification suggests that $\pc_x^k(n,m)$ may be placed in bijection with some subset of $\svt(\lambda/\mu,\rho)$ for $\lambda=(n+\frac{x+k-a-1}{k-1},n+\frac{x+k-a-1}{k-1},m)$ and $\mu=(\frac{x+k-a-1}{k-1},0,0)$.  The difficulty is in describing what subset of $\svt(\lambda/\mu,\rho)$ corresponds to left-justified prographs in which the $k-a$ leftmost children of the initial coproduct terminate at the final product node.

Conjecture \ref{thm: x not 1 mod(k-1) conjecture} describes the subset of $\svt(\lambda/\mu,\rho)$ that should lie in bijection with $\pc_x^k(n,m)$.  The first condition below prevents the $k-a$ leftmost children of the initial coproduct from terminating at a coproduct node.  The second condition prevents those same edges from serving as an input to a product that isn't the final product of the prograph.

\begin{conjecture}
\label{thm: x not 1 mod(k-1) conjecture}
Assume $x \equiv a \mod(k-1)$, where $2 \leq a \leq k-1$.  Then consider $\pc_x^k(n,m)$ and $\svt(\lambda/\mu,\rho)$ with $\lambda=(n+\frac{x+k-a-1}{k-1},n+\frac{x+k-a-1}{k-1},m)$ and $\mu=(\frac{x+k-a-1}{k-1},0,0)$.  For arbitrary $T \in \svt(\lambda/\mu,\rho)$, let $b_1 < b_2 < \hdots$ denote the middle-row entries of $T$ and let $c_1 < c_2 < \hdots$ denote the bottom-row entries of $T$.  Then $\pc_x^k(n,m)$ is in bijection with the subset of tableaux from $\svt(\lambda/\mu,\rho)$ satisfying
\begin{enumerate}
\item $b_i = i$ for all $1 \leq i \leq k-a$, and
\item $c_i > b_{(k-1)i + 2 - (k-a)}$ for all $1 \leq i \leq m-1$
\end{enumerate}
\end{conjecture}

\subsection{Additional Combinatorial Interpretations for $\svt(n^3,\rho)$ with $\rho = (1,k-1,1)$}
\label{subsec: additional interpretations}

It is natural to suppose that all combinatorial interpretations of the three-dimensional Catalan numbers admit one-parameter generalizations that lie in bijection with $\svt(n^3,\rho)$ for $\rho = (1,k-1,1)$.  Below we briefly conjecture as to how several more of those interpretations may be $k$-generalized.  See sequence A005789 of OEIS \cite{OEIS} for a full list of candidates.  Beyond the interpretations discussed below, we are especially interested in how the pattern-avoiding permutations of Lewis \cite{Lewis} may be generalized using standard set-valued Young tableaux.

\begin{enumerate}
\item The three-dimensional Catalan number $C_{3,n}$ is known to count the number of walks in the first quadrant of $\Z^2$ that start and end at $(0,0)$ and use $3n$ total steps from $\lbrace (0,1),(1,-1), (-1,0) \rbrace$.  These walks are known to lie in bijection with $S(n^3)$ via a map that associates $(0,1)$ steps with entries in the top row of the corresponding tableau, $(1,-1)$ steps with entries in the middle row of that tableau, and $(-1,0)$ steps with entries in the bottom row of that tableau.  We conjecture that this map may be generalized to a bijection between $\svt(n^3,\rho)$ with $\rho=(1,k-1,1)$ and walks in the first quadrant of $\Z^2$ that start and end at $(0,0)$ and which use $(k+1)n$ total steps from $\lbrace (0,k-1), (1,-1), (-k+1,0) \rbrace$.  In this bijection, $(0,k-1)$ steps should correspond to entries in the top row of the associated set-valued tableau, $(1,-1)$ should correspond to entries in the middle row of that tableau, and $(-k+1,0)$ entries should correspond to entries in the bottom row of that tableau.
\item $C_{3,n}$ is also known to count three-dimensional integer lattice paths from $(0,0,0)$ to $(n,n,n)$ that use steps from $\lbrace (1,0,0), (0,1,0), (0,0,1) \rbrace$ and that satisfy $x \geq y \geq z$ at every lattice point $(x,y,z)$ along the path.  These lattice paths are known to lie in bijection with $S(n^3)$ via a map that associates $(1,0,0)$ steps with entries in the top row of the corresponding tableau, $(0,1,0)$ steps with entries in the middle row of that tableau, and $(0,0,1)$ steps with entries in the bottom row of that tableau.  It should be straightforward to generalize this map to a bijection between $\svt(n^3,\rho)$ with $\rho(1,k-1,1)$ and integer lattice paths from $(0,0,0)$ to $((k-1)n,n,(k-1)n)$ that use steps from $\lbrace (1,0,0), (0,1,0), (1,0,0) \rbrace$ and which satisfy $(k-1)x \geq y \geq (k-1)z$ at every point $(x,y,z)$.  This bijection would similarly associate $(1,0,0)$ steps to top-row entries, $(0,1,0)$ to middle-row entries, and $(0,0,1)$ to bottom-row entries.
\end{enumerate}

For a somewhat different application of set-valued tableaux with $\rho = (1,k-1,1)$, we refer the reader to the work of Eu \cite{Eu}.  Eu places all standard Young tableaux with at most three rows and any shape $\lambda \vdash N$ in bijection with Motzkin paths of length $n$.  By Motzkin paths of length $n$ we mean integer lattice paths from $(0,0)$ to $(n,0)$ that use steps from $\lbrace (1,1),(1,-1),(1,0) \rbrace$ and never fall below the $x$-axis.

Direct computations for small $n$ reveal that a similar result may hold for standard set-valued Young tableaux with at most three rows, precisely $n(k-1)$ entries, and densities (determined by the number of rows) of either $\rho_1 = (1)$, $\rho_2 = (1,k-1)$, or $\rho = (1,k-1,1)$.  In particular, such tableaux appear to lie in bijection with what we refer to as $(k-1)$-sloped Motzkin paths of length $n$: lattice paths from $(0,0)$ to $(n,0)$ that use steps from $\lbrace (k-1,1),(1,-1),(1,0) \rbrace$ and which never fall below the $x$-axis.  The only caveat here is that one cannot include tableaux with ``partially filled" cells: every cell must have the full complement of entries determined by $\rho_i$.\footnote{$k$-sloped Motzkin paths should not be confused with the pre-existing notion of $k$-Motzkin paths, which correspond to $2$-sloped Motzkin paths in which every horizontal steps carries one of $k$ colors.  See Barrucci, Del Lungo, Pergola and Pinazni \cite{BDPP} for a treatment of $k$-Motzkin paths}

See Figure \ref{fig: motzkin paths} for a comparison of $3$-sloped Motzkin paths of length $n=4$ and set-valued tableaux with density from $\lbrace (1), (1,2), (1,2,1) \rbrace$ and precisely $4$ entries.  For justification of the specific matching exhibited in Figure \ref{fig: motzkin paths}, we direct the reader to the algorithm presented by Eu \cite{Eu}.

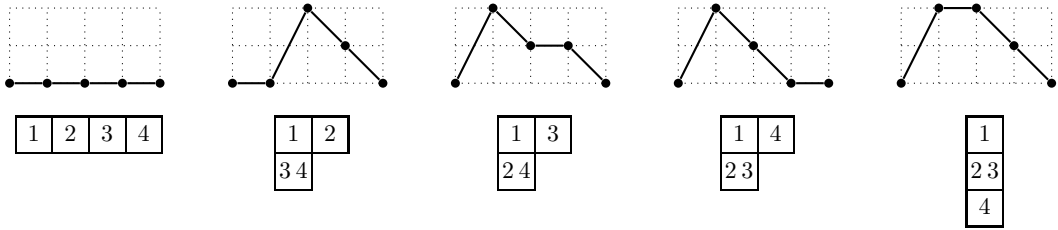
\begin{figure}[ht!]
\centering
\ytableausetup{boxsize=1.5em}

\setlength{\tabcolsep}{12pt}
\begin{tabular}{c c c c c}
\begin{tikzpicture}
	[scale=.5,auto=left,every node/.style={circle, fill=black,inner sep=1.2pt}]
	\draw[dotted] (0,0) to (4,0);
	\draw[dotted] (0,1) to (4,1);
	\draw[dotted] (0,2) to (4,2);
	\draw[dotted] (0,0) to (0,2);
	\draw[dotted] (1,0) to (1,2);
	\draw[dotted] (2,0) to (2,2);
	\draw[dotted] (3,0) to (3,2);
	\draw[dotted] (4,0) to (4,2);
	\node (0*) at (0,0) {};
	\node (1*) at (1,0) {};
	\node (2*) at (2,0) {};
	\node (3*) at (3,0) {};
	\node (4*) at (4,0) {};
	\draw[thick] (0*) to (1*);
	\draw[thick] (1*) to (2*);
	\draw[thick] (2*) to (3*);
	\draw[thick] (3*) to (4*);
\end{tikzpicture}
&
\begin{tikzpicture}
	[scale=.5,auto=left,every node/.style={circle, fill=black,inner sep=1.2pt}]
	\draw[dotted] (0,0) to (4,0);
	\draw[dotted] (0,1) to (4,1);
	\draw[dotted] (0,2) to (4,2);
	\draw[dotted] (0,0) to (0,2);
	\draw[dotted] (1,0) to (1,2);
	\draw[dotted] (2,0) to (2,2);
	\draw[dotted] (3,0) to (3,2);
	\draw[dotted] (4,0) to (4,2);
	\node (0*) at (0,0) {};
	\node (1*) at (1,0) {};
	\node (2*) at (2,2) {};
	\node (3*) at (3,1) {};
	\node (4*) at (4,0) {};
	\draw[thick] (0*) to (1*);
	\draw[thick] (1*) to (2*);
	\draw[thick] (2*) to (3*);
	\draw[thick] (3*) to (4*);
\end{tikzpicture}
&
\begin{tikzpicture}
	[scale=.5,auto=left,every node/.style={circle, fill=black,inner sep=1.2pt}]
	\draw[dotted] (0,0) to (4,0);
	\draw[dotted] (0,1) to (4,1);
	\draw[dotted] (0,2) to (4,2);
	\draw[dotted] (0,0) to (0,2);
	\draw[dotted] (1,0) to (1,2);
	\draw[dotted] (2,0) to (2,2);
	\draw[dotted] (3,0) to (3,2);
	\draw[dotted] (4,0) to (4,2);
	\node (0*) at (0,0) {};
	\node (1*) at (1,2) {};
	\node (2*) at (2,1) {};
	\node (3*) at (3,1) {};
	\node (4*) at (4,0) {};
	\draw[thick] (0*) to (1*);
	\draw[thick] (1*) to (2*);
	\draw[thick] (2*) to (3*);
	\draw[thick] (3*) to (4*);
\end{tikzpicture}
&
\begin{tikzpicture}
	[scale=.5,auto=left,every node/.style={circle, fill=black,inner sep=1.2pt}]
	\draw[dotted] (0,0) to (4,0);
	\draw[dotted] (0,1) to (4,1);
	\draw[dotted] (0,2) to (4,2);
	\draw[dotted] (0,0) to (0,2);
	\draw[dotted] (1,0) to (1,2);
	\draw[dotted] (2,0) to (2,2);
	\draw[dotted] (3,0) to (3,2);
	\draw[dotted] (4,0) to (4,2);
	\node (0*) at (0,0) {};
	\node (1*) at (1,2) {};
	\node (2*) at (2,1) {};
	\node (3*) at (3,0) {};
	\node (4*) at (4,0) {};
	\draw[thick] (0*) to (1*);
	\draw[thick] (1*) to (2*);
	\draw[thick] (2*) to (3*);
	\draw[thick] (3*) to (4*);
\end{tikzpicture}
&
\begin{tikzpicture}
	[scale=.5,auto=left,every node/.style={circle, fill=black,inner sep=1.2pt}]
	\draw[dotted] (0,0) to (4,0);
	\draw[dotted] (0,1) to (4,1);
	\draw[dotted] (0,2) to (4,2);
	\draw[dotted] (0,0) to (0,2);
	\draw[dotted] (1,0) to (1,2);
	\draw[dotted] (2,0) to (2,2);
	\draw[dotted] (3,0) to (3,2);
	\draw[dotted] (4,0) to (4,2);
	\node (0*) at (0,0) {};
	\node (1*) at (1,2) {};
	\node (2*) at (2,2) {};
	\node (3*) at (3,1) {};
	\node (4*) at (4,0) {};
	\draw[thick] (0*) to (1*);
	\draw[thick] (1*) to (2*);
	\draw[thick] (2*) to (3*);
	\draw[thick] (3*) to (4*);
\end{tikzpicture}

\vspace{.1in}
\\
\scalebox{.9}{
\begin{ytableau}
1 & 2 & 3 & 4
\end{ytableau}}
&
\scalebox{.9}{
\begin{ytableau}
1 & 2 \\
3 \kern+2pt 4 
\end{ytableau}}
&
\scalebox{.9}{
\begin{ytableau}
1 & 3 \\
2 \kern+2pt 4 
\end{ytableau}}
&
\scalebox{.9}{
\begin{ytableau}
1 & 4 \\
2 \kern+2pt 3 
\end{ytableau}}
&
\scalebox{.9}{
\begin{ytableau}
1 \\
2 \kern+2pt 3 \\
4 
\end{ytableau}}
\end{tabular}

\caption{$3$-sloped Motzkin paths of length $4$ and standard set-valued Young tableaux with $4$ entries across at most three-rows and densities of either $\rho_1 = (1)$, $\rho_2 = (1,2)$, or $\rho_3 = (1,2,1)$.}
\label{fig: motzkin paths}
\end{figure}

\subsection{$\svt(\lambda,\rho)$ for Distinct Three- and Four-Row Densities}
\label{subsec: other sets}

We close this paper by briefly exploring several additional densities for standard set-valued Young tableaux of shapes $\lambda = n^3$ and $\lambda = n^4$.  The cardinalities of the resulting sets $\svt(\lambda,\rho)$ correspond to one-parameter generalizations of the three- and four-dimensional Catalan numbers that are distinct from the three-dimensional $k$-Catalan numbers $C_{3,n}^k$ of previous sections.  It is our hope that combinatorial interpretations as interesting as those for $C_{3,n}^k$ will eventually be found for each of these generalizations.

First consider the case of $\lambda = n^3$ and $\widetilde{\rho} = (k-1,1,1)$, where $k \geq 1$.  We informally refer to the resulting integers $\widetilde{C}_{3,n}^k = \vert \svt(n^3,\widetilde{\rho}) \vert$ as the non-involutory three-dimensional $k$-Catalan numbers.  This title is motivated by the fact that the set-valued Sch\"utzenberger involution is no longer an automorphism of $\svt(n^3,\widetilde{\rho})$ but a bijection onto the distinct set $\svt(n^3,\widetilde{\rho}')$ with $\widetilde{\rho}'=(1,1,k-1)$.  Observe from Tables \ref{tab: 1,k-1,1} and \ref{tab: k-1,1,1} of Appendix \ref{sec: appendix} that $\widetilde{C}_{3,n}^k \leq C_{3,n}^k$ for all choices of $n,k$ where both values are known.

Applying the methods of Section \ref{sec: enumeration} to $\svt(\lambda,\widetilde{\rho})$ yields the closed formulas of Proposition \ref{thm: k-1,1,1 values} and the general recurrences of Proposition \ref{thm: k-1,1,1 recurrences}.  See Table \ref{tab: k-1,1,1} of Appendix \ref{sec: appendix} for all known values of $\widetilde{C}_{3,n}^k = \vert \svt(n^3,\widetilde{\rho}) \vert$.

Pause to note that the recurrences of Proposition \ref{thm: k-1,1,1 recurrences} are significantly harder to apply than those for $\rho = (1,k-1,1)$ that appear in Proposition \ref{thm: general n recurrences}, as the recurrences of Proposition \ref{thm: k-1,1,1 recurrences} involve enumerations of (non-set-valued) standard skew Young tableaux.  This is a difficulty that appears to extend to all three- (and four-) row densities other than $\rho = (1,k-1,1)$.

\begin{proposition}
\label{thm: k-1,1,1 values}
Let $\widetilde{\rho} = (k-1,1,1)$.  For any $k \geq 1$,

$$\widetilde{C}^k_{3,2} = \vert \svt(2^3,\widetilde{\rho}) \vert = \frac{1}{2}k^2 + \frac{3}{2}k$$

$$\widetilde{C}^k_{3,3} = \vert \svt(3^3,\widetilde{\rho}) \vert = \frac{2}{3}k^4 + 3 k^3 + \frac{7}{3}k^2 - k$$

$$\widetilde{C}^k_{3,4} = \vert \svt(4^3,\widetilde{\rho}) \vert = \frac{25}{18}k^6 + \frac{61}{8}k^5 + \frac{175}{18}k^4 - \frac{35}{24}k^3 - \frac{37}{9}k^2 + \frac{5}{6}k$$
\end{proposition}

\begin{proposition}
\label{thm: k-1,1,1 recurrences}
Fix $k \geq 1$.  For $\widetilde{\rho} = (k-1,1,1)$ and any three-row shape $\lambda = (a,b,c)$ with $a \leq b \leq c$,

$$|\svt((a,b,c),\widetilde{\rho}) | =
\begin{cases}
\displaystyle{\sum_{\substack{0 \leq j \leq i \leq b,\\[1pt]j \leq c}} \binom{b-i+c-j+k-2}{k-2} \kern+2pt \vert S((b,c)/(i,j)) \vert \cdot \vert \svt((a-1,i,j),\widetilde{\rho}) \vert}, & \text{if $a > b$;}\\[22pt]
\displaystyle{\sum_{1 \leq i \leq c} \vert \svt((a,b-1,i),\widetilde{\rho}) \vert}, & \text{if $a = b > c$;}\\[22pt]
\displaystyle{\vert \svt((a,b,c-1),\widetilde{\rho}) \vert}, & \text{if $a=b=c$.}
\end{cases}$$
\end{proposition}

In the case of $\lambda = n^4$, we recognize the densities $\xi_i = (1,k-1,k-1,1)$ and $\xi_2 = (k-1,1,1,1)$ as prime candidates to obtain what should be referred to as the (involutory) four-dimensional $k$-Catalan numbers $C_{4,n}^k = \vert \svt(4^n,\xi_1) \vert$ and the non-involutory four-dimensional $k$-Catalan numbers $\widetilde{C}_{4,n}^k = \vert \svt(4^n,\xi_2) \vert$.  As the addition of a fourth row makes the techniques of Section \ref{sec: enumeration} significantly harder to apply, we simply direct the reader to Tables \ref{tab: 1,k-1,k-1,1} and Table \ref{tab: k-1,1,1,1} of Appendix \ref{sec: appendix} for all known values of $C_{4,n}^k = \vert \svt(4^n,\xi_1) \vert$ and $\widetilde{C}_{4,n}^k = \vert \svt(4^n,\xi_2) \vert$.

\appendix

\newpage

\section{Tables of Values}
\label{sec: appendix}
Values were obtained via a combination of proven results (Section \ref{sec: enumeration}, Subsection \ref{subsec: other sets}) and direct enumeration in Java.  Java coding was performed by Benjamin Levandowski of Valparaiso University and is available upon request.

\begin{table}[ht!]
\centering
\caption{Known values of $C^k_{3,n} = \vert \svt(n^3,\rho) \vert$ for $\rho=(1,k-1,1)$}
\label{tab: 1,k-1,1}
\small
\begin{tabular}{|>{$}c<{$}|>{$}c<{$} >{$}c<{$} >{$}c<{$} >{$}c<{$} >{$}c<{$} >{$}c<{$}|}
\hline
k \backslash n & 1 & 2 & 3 & 4 & 5 & 6 \\ \hline
1 & 1 & 2 & 5 & 14 & 42 & 132 \\
2 & 1 & 5 & 42 & 462 & 6006 & 87516 \\
3 & 1 & 10 & 190 & 4295 & 153415 & 5396601 \\
4 & 1 & 17 & 581 & 27461 & 1566018 & 100950800 \\
5 & 1 & 26 & 1401 & 105026 & 9511451 & \\
6 & 1 & 37 & 2890 & 315014 & 41500117 & \\
7 & 1 & 50 & 5342 & 797917 & 144067106 & \\ \hline
\end{tabular}
\end{table}

\begin{table}[ht!]
\centering
\caption{Known values of $\widetilde{C}^k_{3,n} = \vert \svt(n^3,\widetilde{\rho}) \vert$ for $\widetilde{\rho}=(k-1,1,1)$}
\label{tab: k-1,1,1}
\small
\begin{tabular}{|>{$}c<{$}|>{$}c<{$} >{$}c<{$} >{$}c<{$} >{$}c<{$} >{$}c<{$} >{$}c<{$}|}
\hline
k \backslash n & 1 & 2 & 3 & 4 & 5 & 6 \\ \hline
1 & 1 & 2 & 5 & 14 & 42 & 132 \\
2 & 1 & 5 & 42 & 462 & 6006 & 87516 \\
3 & 1 & 9 & 153 & 3579 & 101630 & 3288871 \\
4 & 1 & 14 & 396 & 15830 & 779063 & 44072801 \\
5 & 1 & 20 & 845 & 51325 & 3872370 & \\
6 & 1 & 27 & 1590 & 136234 & 14589623 & \\
7 & 1 & 35 & 2737 & 314202 & & \\ \hline
\end{tabular}
\end{table}

\begin{table}[ht!]
\centering
\caption{Known values of $\vert \svt(n^4,\xi_1) \vert$ for $\xi_1=(1,k-1,k-1,1)$}
\label{tab: 1,k-1,k-1,1}
\small
\begin{tabular}{|>{$}c<{$}|>{$}c<{$} >{$}c<{$} >{$}c<{$} >{$}c<{$} >{$}c<{$} >{$}c<{$}|}
\hline
k \backslash n & 1 & 2 & 3 & 4 & 5 & 6 \\ \hline
1 & 1 & 2 & 5 & 14 & 42 & 132 \\
2 & 1 & 14 & 462 & 24024 & 1662804 & 140229804 \\
3 & 1 & 84 & 24521 & 13074832 & & \\
4 & 1 & 460 & 960875 & 3959335892 & & \\
5 & 1 & 2380 & 31378194 & & & \\
6 & 1 & 11814 & & & & \\
7 & 1 & 57288 & & & & \\ \hline
\end{tabular}
\end{table}

\begin{table}[ht!]
\centering
\caption{Known values of $\vert \svt(n^4,\xi_2) \vert$ for $\xi_2=(k-1,1,1,1)$}
\label{tab: k-1,1,1,1}
\small
\begin{tabular}{|>{$}c<{$}|>{$}c<{$} >{$}c<{$} >{$}c<{$} >{$}c<{$} >{$}c<{$} >{$}c<{$}|}
\hline
k \backslash n & 1 & 2 & 3 & 4 & 5 & 6 \\ \hline
1 & 1 & 5 & 42 & 462 & 6006 & 87516 \\
2 & 1 & 14 & 462 & 24024 & 1662804 & 140229804 \\
3 & 1 & 28 & 2158 & 281571 & 50972547 & \\
4 & 1 & 48 & 6990 & 1798860 & 658138000 & \\
5 & 1 & 75 & 18275 & 8103935 & & \\
6 & 1 & 110 & 41382 & 28950168 & & \\
7 & 1 & 154 & 84427 & & & \\ \hline
\end{tabular}
\end{table}

\end{document}